\documentclass[11pt]{article}
\usepackage[dvips]{graphicx}
\usepackage{amssymb,amsmath,amsthm,color}
\usepackage{url}
\usepackage{natbib}
\bibpunct{(}{)}{;}{a}{,}{,}
 \usepackage[colorlinks,
            linkcolor=blue,
            citecolor=blue,
            urlcolor=magenta,
            linktocpage,
            plainpages=false]{hyperref}

\usepackage{times}
\usepackage{bbm}

\newcommand\tnc{\theta_{\Sigma}}
\newcommand\EE{\mathbb E} 
\newcommand\dom{\text{dom }} 
\newcommand\NN{\mathbb{N}} 
\newcommand\RR{\mathbb{R}} 
\newcommand{\BEAS}{\begin{eqnarray*}}
\newcommand{\EEAS}{\end{eqnarray*}}
\newcommand{\BEA}{\begin{eqnarray}}
\newcommand{\EEA}{\end{eqnarray}}
\newcommand{\BEQ}{\begin{equation}}
\newcommand{\EEQ}{\end{equation}}
\newcommand{\BIT}{\begin{itemize}}
\newcommand{\EIT}{\end{itemize}}
\newcommand{\BNUM}{\begin{enumerate}}
\newcommand{\ENUM}{\end{enumerate}}
\newcommand{\BA}{\begin{array}}
\newcommand{\EA}{\end{array}}
\newcommand{\diag}{\mathop{\rm diag}}
\newcommand{\Diag}{\mathop{\rm Diag}}

\newcommand{\argmin}{\mathop{\rm argmin}}
\newcommand{\argmax}{\mathop{\rm argmax}}

\newcommand{\Tr}{\mathop{ \rm tr}}
\newcommand{\tr}{\mathop{ \rm tr}}
\newcommand{\sign}{\mathop{ \rm sign}}
\newcommand{\idm}{I}
\newcommand{\mysec}[1]{Section~\ref{sec:#1}}
\newcommand{\myapp}[1]{Appendix~\ref{app:#1}}

\newcommand{\eq}[1]{Eq.~(\ref{eq:#1})}
\newcommand{\myfig}[1]{Figure~\ref{fig:#1}}
\newtheorem{lemma}{Lemma}					        
\newtheorem{theorem}{Theorem}					        
\newtheorem{corollary}{Corollary}					        
\newtheorem{proposition}{Proposition}
\newtheorem{definition}{Definition}	
	
\makeatletter
\newcommand{\LeftEqNo}{\let\veqno\@@leqno}
\makeatother

\title{Stochastic Composite Least-Squares Regression\\ with convergence rate $O(1/n)$}

\author{
 Nicolas Flammarion and Francis Bach\\
INRIA - Sierra project-team\\
D\'epartement d'Informatique de l'Ecole Normale Sup\'erieure \\
Paris, France \\
 \texttt{nicolas.flammarion@ens.fr}, \texttt{francis.bach@ens.fr} 
}
\begin{document}

\maketitle

\begin{abstract}
We consider the minimization of composite objective functions composed of the expectation of quadratic functions and an arbitrary convex function. We study the stochastic dual averaging algorithm with a constant step-size, showing that it leads to a convergence rate of $O(1/n)$ without strong convexity assumptions. This thus extends earlier results on least-squares regression with the Euclidean geometry to (a) all convex regularizers and constraints, and (b) all geometries represented by a Bregman divergence. This is achieved by a new proof technique that relates stochastic and deterministic recursions.
\end{abstract}

\section{Introduction}

Many learning problems may be cast as the optimization of an objective function defined as an expectation of random functions, and which can be accessed only through samples. In this paper, we consider \emph{composite} problems of the form 
\BEQ
\min_{\theta \in {\RR^d} } \EE_z \ell(z,\theta)+g(\theta),
\EEQ
where   for any $z$, $\ell(z,\cdot)$ is a convex quadratic function (plus some linear terms) and $g$ is any extended-value convex function.

In a machine learning context, $\ell(z,\theta)$ is the loss occurred for the observation $z$ and the predictor parameterized by $\theta$, $f(\theta) = \EE_z\ell(z,\theta)$ is its generalization error, while the function $g$ represents some additional regularization or constraints on the predictor. Thus in this paper we consider composite least-squares regression problems, noting that solving such problems effectively leads to efficient algorithms for all smooth losses by using an online Newton algorithm~\citep{bm1}, with the same running-time complexity of $O(d)$ per iteration for linear predictions.

When $g=0$,   averaged stochastic gradient descent with a constant step-size achieves the optimal  convergence rate of $O(1/n)$ after $n$ observations, even in ill-conditioned settings without strong convexity~\citep{DieFlaBac16,JaiKak16}, with precise non-asymptotic results that depend on the statistical noise variance  $\sigma^2$ of the least-squares problem, as $  {\sigma^2 d}/{n}$, and on the squared Euclidean distance between the initial predictor $\theta_0$ and the optimal predictor $\theta_\ast$, as $ {\| \theta_0 - \theta_\ast\|_2^2}/{n}$.

In this paper, we extend this $O(1/n)$ convergence result in two different ways: 
\BIT

\vspace*{-.2cm}

\item \textbf{Composite problems}: we provide a new algorithm that deals with composite problems where~$g$ is (essentially) any extended-value convex function, such as the indicator function of a convex set for constrained optimization, or a norm or squared norm for additional regularization. This situation is common in many applications in machine learning and signal processing~\citep[see, e.g.,][and references therein]{rish2014sparse}. Because we consider large steps-sizes (that allow robustness to ill-conditioning), the new algorithm is \emph{not} simply a proximal extension; for example, in the constrained case, averaged projected stochastic gradient descent with a constant step-size is not convergent, even for quadratic functions.

\vspace*{-.2cm}

\item \textbf{Beyond Euclidean geometry}: Following mirror descent~\citep{Nem79}, our new algorithm can take into account a geometry obtained with a Bregman divergence $D_h$ associated with a convex function $h$, which can typically be  the squared Euclidean norm (leading to regular stochastic gradient descent in the non-composite case), the entropy function, or  the squared $\ell_p$-norm. This will allow convergence rates proportional to $  D_h(\theta_\ast,\theta_0)/n$, which may be significantly smaller than $ {\| \theta_0 - \theta_\ast\|^2}/{n}$ in many situations.
\EIT

In order to obtain these two extensions, we consider the stochastic dual averaging algorithm of~\citet{Nes07} and \citet{xia10} which we present in \mysec{da}, and   study under the particular set-up of \emph{constant step-size with averaging}, showing in \mysec{stogen} that it also achieves a convergence rate of $O(1/n)$ even without strong-convexity. This is achieved by a new proof technique that relates stochastic and deterministic recursions.

Given that known lower-bounds for this class of problems are proportional to $1/\sqrt{n}$ for function values, we established our $O(1/n)$ results with a different criterion, namely the Mahalanobis distance associated with the Hessian of the least-squares problem. In our simulations in \mysec{simulations}, the two criteria behave similarly. Finally, in \mysec{mdda}, we shed additional insights of the relationships between mirror descent and dual averaging, 
 in particular in terms of  continuous-time interpretations.
 
\section{Dual averaging algorithm}

In this section, we introduce  dual averaging  as well as related frameworks, together with new results in the deterministic case.

\label{sec:da}
\subsection{Assumptions}
We consider the Euclidean space ${\RR^d}$ of dimension $d$ endowed with the natural inner product $\langle \cdot,\cdot \rangle$ and an arbitrary norm $\Vert \cdot\Vert$ (which may not be the Euclidean norm). We denote by $\Vert \cdot \Vert_*$ its dual norm and for any  symmetric positive-definite matrix~$A$, by  $\Vert \cdot \Vert_{A}=\sqrt {\langle \cdot, A\cdot \rangle}$ the Mahalanobis norm. For a vector $\theta\in{\RR^d}$, we denote by $\theta(i)$ its $i$-th coordinate   and by $\Vert \theta\Vert_p=(\sum_{i=1}^d \vert\theta(i)\vert^p)^{1/p}$ its $\ell_p$-norm. We also denote the convex conjugate of a function $f$  by $f^*(\eta)=\sup_{\theta\in{\RR^d}}  \langle \eta,\theta\rangle-f(\theta) $. 
We remind that a function $f$ is $L$-smooth with respect to a norm $\Vert\cdot\Vert$ if for all $(\alpha,\beta)\in{\RR^d} \times \RR^d$, $\Vert \nabla f(\alpha)-\nabla f(\beta)\Vert_*\leq L \Vert \alpha-\beta\Vert$ and is  $\mu$-strongly convex if  for all $(\alpha,\beta)\in{\RR^d} \times \RR^d$ and $g\in\partial f (\beta)$, $f(\alpha)\geq f(\beta)+\langle g,\alpha-\beta\rangle +\frac{\mu}{2}\Vert \alpha-\beta\Vert^2$ \citep[see, e.g.,][]{ShaSin06}.  

We consider problems of the form:
 \begin{equation}\label{eq:comp}
 \min_{\theta\in\mathcal{\mathcal{X}}} \psi(\theta)= f(\theta)+g(\theta), 
\end{equation}
where  $\mathcal{X}\subset {\RR^d}$ is a closed convex set with non empty interior. Throughout this paper, we make the following general assumptions:
\begin{description}

\vspace*{-.1444cm}

 \item[(A1)] $f:{\RR^d} \to \RR\cup\{+\infty\}$ is a proper lower semicontinuous convex function and is differentiable on $ \mathring{\mathcal X}$ (the interior of $\mathcal X$).

\vspace*{-.1444cm}

 \item[(A2)] $g:{\RR^d} \to \RR\cup\{+\infty\}$ is a proper lower semicontinuous convex function.

\vspace*{-.1444cm}

 \item[(A3)]  $h:{\RR^d} \to \RR\cup\{+\infty\}$ with $ \overline{\dom h}\cap \overline{\dom g}=\mathcal{X}$, $\mathring \dom h \cap  \dom g \neq \emptyset$.
Moreover $h$  is a Legendre function \citep[][chap. 26]{Roc70}:
\begin{itemize}

\vspace*{-.1444cm}

 \item $h$ is a proper lower semicontinuous  strictly convex  function, differentiable on $\mathring  \dom h$.
 \item The gradient of $h$ is diverging on the boundary of $\dom h$ (i.e., $ {\lim_{{n\to +\infty}}}\Vert \nabla h (\theta_n)\Vert =\infty$ for any sequence $(\theta_n)$ converging to a boundary point of $\dom h$). Note that $\nabla h$ is then a bijection from $\mathring{\dom} h $ to $\mathring{\dom} h^*$ whose inverse is the gradient of the conjugate $\nabla h^*$.
\end{itemize}

\vspace*{-.1844cm}

\item [(A4)]
 The function $\psi=f+g$ attains its minimum over $\mathcal{X}$ at a certain  $\theta_*\in {\RR^d}$ (which may not be unique).
\end{description}
Note that we adopt the same framework as \citet{BauBolTeb15} with the difference that the convex constraint $\mathcal{C}$ can be handled with more flexibility: either by considering a Legendre function~$h$ whose domain is $\mathcal{C}$ or by considering the hard constraint  $g(\theta)=\mathbbm{1}_C(\theta)$ (equal to $0$ if $\theta \in \mathcal{C}$ and $+\infty$ otherwise).

\subsection{Dual averaging algorithm}

In this section we present the dual averaging algorithm  (referred to from now on as ``DA'')  for solving composite problems of the form of \eq{comp}. It starts from $\theta_0\in \mathring \dom h $ and $\eta_0=\nabla h(\theta_0)$ and   iterates  for $n\geq 1$ the recursion
\BEA\label{eq:da}
\eta_n&=& \eta_{n-1}-\gamma \nabla f(\theta_{n-1}) \nonumber\\
\theta_n&=& \nabla h_n^*(\eta_n),
\EEA
with $h_n=h+n\gamma g$ and $\gamma \in (0,\infty)$ (commonly referred to as the step-size in optimization or the learning rate in machine learning). We note that equivalently $\theta_n\in\argmax_{\theta\in{\RR^d}}\{ \langle\eta_n,\theta\rangle-h_n(\theta)\}$. When $h= \frac{1}{2} \| \cdot \|_2^2$ and $g=0$, we recover gradient descent.

Two iterates $(\eta_n,\theta_n)$ are updated in DA. The dual iterate $\eta_n$ is simply proportional to the sum of the gradients evaluated in the primal iterates $(\theta_n)$. The update of the primal iterate $\theta_n$ is more complex and raises two different issues: its existence and its tractability. We discuss the first point in  \myapp{dadefined} and assume, as of now, that the method is generally well defined  in practice. The tractability of $\theta_n$ is essential and the algorithm is only used in practice if the functions~$h$ and $g$ are simple in the sense that the gradient $\nabla h_n^*$ may be computed effectively. This is the case if there exists a closed form expression. Usual examples are given in \myapp{examples}.

 \paragraph{Euclidean case and proximal operators.}
In the Euclidean case, \eq{da} may be written in term of the proximal operator defined by \citet{Mor62} as  $\text{Prox}_g(\eta)=\argmin_{\theta\in \mathcal{X}}\{\frac{1}{2}\Vert \theta-\eta\Vert_2^2+g(\theta)\}$:
\[
 \theta_n
 =
 \argmin_{\theta\in \mathcal{X}}\Big\{ \langle -\eta_n,\theta\rangle +n\gamma g(\theta) +\frac{1}{2}\Vert \theta\Vert_2^2\Big\}
 =
 \argmin_{\theta\in \mathcal{X}}\Big\{\frac{1}{2}\Vert \theta-\eta_n\Vert_2^2+n\gamma g(\theta) \Big\}= \text{Prox}_{\gamma n g}(\eta_n).
\]
DA is in this sense related to proximal gradient methods, also called forward-backward splitting methods \citep[see, e.g., ][]{bt,Wri09,ComPes11}. These methods are tailored  to composite optimization problems: at each iteration   $f$ is linearized around the current iterate $\theta_n$ and they consider the following update
\[
 \theta_{n+1}=\argmin_{\theta\in\mathcal{X}}\Big\{\langle \gamma \nabla  f(\theta_n),\theta\rangle +\gamma g(\theta)+\frac{1}{2}\Vert \theta-\theta_n\Vert_2^2\Big\}
 =
 \text{Prox}_{\gamma g}(\theta_n-\gamma \nabla f(\theta_n)).
\]
Note the difference with DA which considers a dual iterate and a proximal operator for the function $n \gamma g$ instead of $\gamma g$ (see additional insights in \mysec{mdda}).

\paragraph{From non-smooth to smooth optimization.}
DA was  initially introduced by \citet{Nes07} to optimize a non-smooth function  $f$ with possibly convex constraints ($g=0$ or $g=\mathbbm{1}_{\mathcal{C}}$). It was extended to the general stochastic composite case by \citet{xia10}  who defined the iteration as
\[
\theta_{n}=\argmin_{\theta\in \mathcal{X}}\Big\{  \frac{1}{n}\sum_{i=0}^{n-1}\langle z_{i},\theta \rangle + g(\theta) +  \frac{\beta_{n}}{n} h(\theta) \Big\},
\]
where $z_{i}$ is an unbiased estimate\footnote{Their results remain true in the more general setting of online learning.} of a subgradient in $\partial f(\theta_{i})$ and  $(\beta_n)_{n\geq1}$ a nonnegative and nondecreasing sequence of real numbers. This formulation is equivalent to \eq{da} for constant sequences $\beta_{n}=1/\gamma$. \citet{xia10} proved convergence rates of order $O(1/\sqrt{n})$ for convex problems with decreasing step-size $C/\sqrt{n}$ and $O(1/(\mu{n}))$ for problems with $\mu$-strongly convex regularization with  constant step-size $1/\mu$.   DA was also studied with decreasing step-sizes  in the distributed case by \citet{DucAga12,Dek12,COl16} and combined with the alternating direction method of multipliers (ADMM) by \citet{Suz13}. It was further shown to be very efficient in manifold identification by \citet{LeeWri12} and \citet{DucRua16}.

\paragraph{Relationship with mirror descent.}

The DA method should be associated with its cousin  mirror descent algorithm (referred to from now on as ``MD''), introduced by \citet{Nem79} for the constrained case and written under its modern proximal form by \citet{Teb03}
\[
\theta_{n}=\argmin_{\theta\in \mathcal{X}}\big\{  \gamma \langle \nabla f (\theta_{n-1}),\theta \rangle  + D_{h}(\theta,\theta_{n-1})\big\},
\]
where we denote by $D_h(\alpha,\beta)=h(\alpha)-h(\beta)-\langle\nabla h(\beta),\alpha-\beta\rangle$ the Bregman divergence associated with $h$. Moreover it was later extended to the general composite case by \citet{Duc10}  
\BEQ\label{eq:mdprox}
\theta_{n}=\argmin_{\theta\in \mathcal{X}}\big\{  \gamma \langle \nabla f (\theta_{n-1}),\theta \rangle + \gamma g(\theta) + D_{h}(\theta,\theta_{n-1})\big\}.
\EEQ
DA was initially motivated by  \citet{Nes07}  to avoid  new gradients to be taken into account with less weight than previous ones. However, as an extension of the Euclidean case, DA essentially differs from MD  on the way the regularization component is dealt with. See more comparisons in \mysec{mdda}.

\paragraph{Relationship with online learning.}
DA was traditionally studied under the online learning setting \citep{Zin03} of regret minimization and is   related to the ``follow the leader''  approach \citep[see, e.g.,][]{KaiVam05} as noted by \citet{McM11}. More generally, the DA method may be cast in the primal-dual algorithmic framework of \citet{ShaSin06} and \citet{sha09}.

\subsection{Deterministic convergence result for dual averaging}

In this section we present the convergence properties of the DA method for optimizing deterministic composite problems of the form in \eq{comp}, for any smooth function $f$ (see proof in \myapp{deterministic}).
\begin{proposition}\label{prop:phifunction}
 Assume \textbf{(A1-4)}. For any step-size $\gamma$ such that $h-\gamma f$ is convex on $\mathring{ \mathcal{X}}$ we have for all $\theta\in\mathcal{X}$ 
 \[
\psi(\theta_n)-\psi(\theta)\leq\frac{D_h(\theta,\theta_0) }{\gamma(n+1)}.
 \]
Moreover assume $g=0$, and there exists $\mu\in\RR$ such that $f-\mu h$ is also convex on $\mathring{ \mathcal{X}}$ then we have for all $\theta\in\mathcal{X}$ 
 \[
f(\theta_n)-f(\theta)\leq (1-\gamma\mu)^n \frac{D_h(\theta,\theta_0) }{\gamma}.
 \]
\end{proposition}
We can make the following remarks:

\vspace*{-.1444cm}

\begin{itemize}
\item
We adapt the proof of \citet{Teb03} to the composite case and the DA method by including the regularization component $g$ in the Bregman divergence. If $g$ was differentiable we would simply use $D_{h_n} = D_{h + n \gamma g}$  and prove the following recursion:
\begin{multline*}\textstyle
  D_{h_n}(\theta_*,\theta_n)- D_{h_{n-1}}(\theta_*,\theta_{n-1})=- D_{h_{n-1}}(\theta_n,\theta_{n-1}) +\gamma \langle \nabla f(\theta_{n-1}), \theta_{n-1}-\theta_n \rangle
\\
- \gamma (g(\theta_n)-g(\theta)) -\gamma \langle \nabla f(\theta_{n-1}), \theta_{n-1}-\theta_* \rangle.
\end{multline*}
Since $g$ is not differentiable, we extend instead the notion of Bregman divergence to the non-smooth case in \myapp{nonsmoothbreg} and show the proof works in the same way.

\vspace*{-.1444cm}

\item 
\textbf{Related work}: DA was first analyzed for smooth functions in the non-composite case where $g=0$, by \citet{Dek12} in the stochastic setting and  by \citet{LuFreNes16} in the deterministic setting. A result on MD with analogue assumptions is presented by \citet{BauBolTeb15} but depends on a symmetry measure of the Bregman divergence $D_h$ which reflects how  different are $D_h(\alpha,\beta)$ and $D_h(\beta,\alpha)$, and the bound is not as simple. The technique to extend the Bregman divergence to analyze the regularization component has its roots in the time-varying potential method in online learning \citep[Chapter 11.6]{CesLug06} and the ``follow the regularized leader'' approach \citep{AbeHazRak08}.

\vspace*{-.1444cm}

\item
This convergence rate is suboptimal for the class of addressed problems. Indeed accelerated gradient methods achieve the convergence rate of $O(L/n^2)$ in the composite setting \citep{Nes13}, such a rate being optimal for optimizing smooth functions  among first-order techniques that can access only sequences of gradients  \citep{nest2004}.

\vspace*{-.1444cm}

\item
Classical results on the convergence of optimization algorithms in non-Euclidean geometries assume on one hand that the function $h$ is strongly convex and on the other hand the function~$f$ is Lipschitz or smooth. Following \citet{BauBolTeb15,LuFreNes16}, we consider a different assumption which combines the smoothness  of $f$ and the strong convexity of $h$  on the single condition $ h-\gamma f$ convex.
For the Euclidean geometry where $h(\theta) = \frac{1}{2} \| \theta\|_2^2$, this condition is obviously equivalent to the smoothness of the function $f$ with regards to the $\ell_2$-norm. Moreover, under arbitrary norm $\Vert \cdot \Vert$, this is also equivalent to assuming $h$ $\mu$-strongly convex and $f$ $L$-smooth (with respect to this norm).
However it is much more general and may hold even when $f$ is non-smooth, which precisely justifies the introduction of this condition  \citep[see examples described by][]{BauBolTeb15,LuFreNes16}.  

\vspace*{-.1444cm}

\item  The bound adapts to the geometry of the function $h$ through the Bregman divergence between the starting point $\theta_0$ and the solution $\theta_*$ and the step-size $\gamma$ which is controlled by $h$. Therefore the choice of $h$ influences the constant in the bound. Examples are provided in \myapp{examples}.
\end{itemize}

\section{Stochastic convergence results for quadratic functions}\label{sec:stogen}

In this section, we consider  a symmetric positive semi-definite   matrix     $\Sigma\in\RR^{d\times d}$ and a convex quadratic function  $f$ defined as
\begin{equation}\tag{\textbf{A5}}
 f (\theta) =  \textstyle \frac{1}{2}\langle \theta, \Sigma \theta\rangle-\langle q,\theta\rangle, \quad \text{with $q\in \RR^d$ in the column space of $\Sigma$,} \LeftEqNo
\end{equation}
so that $f$ has a global minimizer $\theta_\Sigma \in \RR^d$. Without loss of generality\footnote{
By decomposing $\theta$ in $\theta=\theta_\parallel+\theta_\perp$ with $\theta_\perp\in\mathrm{Null}(\Sigma)$ and $\langle \theta_{\perp},\theta_{\parallel}\rangle=0$ and considering $\psi(\theta)=f(\theta_\parallel)+\tilde g(\theta_\parallel)$ where $\tilde g(\theta_\parallel)=\inf_{\theta_\perp\in\mathrm{Null}(\Sigma)}g(\theta_\perp+\theta_\parallel)$.}, $\Sigma$ is assumed invertible, though its eigenvalues could be arbitrarily small.
The global solution is known to be $\tnc=\Sigma^{-1}q$, but  the inverse of the Hessian is often too expensive to compute when $d$ is large. 
The function may be simply expressed as \mbox{$f(\theta_n)=\frac{1}{2}\langle \theta_n-\tnc,\Sigma(\theta_n-\tnc)\rangle+f(\tnc)$} and the excess of the cost function $\psi=f+g$ as
\begin{multline*}
\psi(\theta_n)-\psi(\theta_*)=\langle \theta_*-\tnc,\Sigma(\theta_n-\theta_*)\rangle +g(\theta_n)-g(\theta_*) \text{ (linear part) }\\
+\frac{1}{2}\langle \theta_n-\theta_*,\Sigma(\theta_n-\theta_*)\rangle \text{ (quadratic part) }.
\end{multline*}
The first-order condition of the optimization problem in \eq{comp} is $0\in \nabla f (\theta_*) + \partial g(\theta_*)$  and  by convexity of $g$ we have $g(\theta_n)-g(\theta_*)\geq \langle z,\theta_n-\theta_*\rangle$ for any $z\in \partial g(\theta_*)$. Therefore this implies that the linear part $g(\theta_n)-g(\theta_*) +\langle \nabla f(\theta_*),\theta_n-\theta_*)\rangle$ is non-negative and we have the bound
\BEQ\label{eq:normlowerfunction}
 \frac{1}{2}\Vert \theta_n-\theta_*\Vert_{\Sigma}^{2}\leq \psi(\theta_n)-\psi(\theta_*).
\EEQ
We derive, in this section, convergence results in terms of the distance $\Vert \theta_n-\theta_*\Vert_{\Sigma}$ which takes into account the  ill-conditioning of the matrix $\Sigma$ and is a lower bound in the excess of function values. Furthermore it directly implies classical results for strongly convex problems. 

In many practical situations, the gradient of $f$ is not available for the recursion in \eq{da}, and we have only access to an unbiased estimate  $\nabla f_{n+1}(\theta_n)$ of the gradient of $f$ at $\theta_n$. We consider in this case the stochastic dual averaging method (referred to from now on as ``SDA'') defined the same way as DA as
\BEA\label{eq:sda}
\eta_n&=& \eta_{n-1}-\gamma \nabla f_n(\theta_{n-1}) \nonumber\\
\theta_n&=& \nabla h_n^*(\eta_n),
\EEA
for  $\theta_0\in \mathring \dom h $ and $\eta_0=\nabla h(\theta_0)$. Here we consider the stochastic approximation framework \citep{kuku}. That is, we let $(\mathcal{F}_n)_{n\geq0}$ be an increasing family of $\sigma$-fields such that for each $\theta\in{\RR^d}$ and for all $n\geq 1$ the random variable $\nabla f_n(\theta)$ is square-integrable and $\mathcal{F}_n$-measurable with $\EE[\nabla f_n(\theta)\vert \mathcal{F}_{n-1}]=\nabla f (\theta)$. This includes (but also extends) the usual machine learning situation where $\nabla f_n$ is the gradient of the loss associated with the $n$-th independent observation. We will consider in the following two different gradient oracles.

\subsection{Additive noise}

We study here the convergence of the SDA recursion defined in \eq{sda} under  an additive noise model:
\begin{description}
 \item [\textbf{(A6)}] For all $n\geq1$, $\nabla f_n(\theta)=\nabla f(\theta)-\xi_n$, where the noise $(\xi_n)_{n\geq1}$ is a square-integrable martingale difference sequence  (i.e., $\EE [ \xi_n | \mathcal{F}_{n-1} ] = 0$) with bounded covariance $\EE[\xi_n\otimes\xi_n]\preccurlyeq C$.
\end{description}
 With this oracle and for the quadratic function $f$, SDA takes the form
\BEA\label{eq:daquaad}
 \eta_n &=&
 \eta_{n-1}-\gamma (\Sigma \theta_{n-1}-q) +\gamma\xi_n \nonumber \\
 \theta_n
 &=&
 \nabla h_n^*(\eta_n).
\EEA
We obtain the following convergence result on the average $\bar \theta_{n}=\frac{1}{n}\sum_{k=0}^{n-1}\theta_{k}$ which is an  extension of results from \citet{bm1} to non-Euclidean geometries and to composite settings  (see proof in \myapp{semisto}). 
\begin{proposition}\label{prop:theta}
 Assume \textbf{(A2-6)}. Consider the recursion in \eq{daquaad} for any constant step-size $\gamma$ such that $h-\gamma f$ is convex. Then 
 \begin{equation*}
  \frac{1}{2}{\EE \Vert  \bar \theta_n - \theta_*\Vert_{\Sigma}^2}
  \leq
  2\min\Bigg\{ {\frac{D_h(\theta_{*},\theta_{0})}{\gamma n}};
  \frac{\Vert \nabla h (\theta_0)-\nabla h (\theta_*)\Vert^2_{\Sigma^{-1}}}{(\gamma n)^2}
  \Bigg\}  
  +\frac{4}{n}\tr \Sigma^{-1}C. 
  \end{equation*}
\end{proposition}

We can make the following observations:

\vspace*{-.1444cm}

\begin{itemize}
 \item
 The proof in the Euclidean case  \citep{bm1} highly uses the equality $\theta_n-\tnc=(I-\gamma \Sigma)(\theta_{n-1}-\tnc)$ which is no longer available in the non-Euclidean or proximal cases. Instead  we adapt the classic proof of convergence of averaged SGD of \citet{pj} which rests upon  the expansion $\sum_{k=0}^{n}\nabla f_{k+1}(\theta_k)=\sum_{k=0}^{n}(\eta_{k}-\eta_{k+1})/\gamma=(\eta_{0}-\eta_{n+1})/\gamma$. The crux of the proof is then to consider the difference between the  iterations with and without noise, $\eta_n^{\text{sto}}-\eta_n^{\text{det}}$, which happens to satisfy a similar recursion as \eq{daquaad} but started from the solution $\theta_*$. The quadratic nature of $f$ is used twice: (a) to bound  $\Vert \eta_n^{\text{sto}}-\eta_n^{\text{det}}\Vert_{\Sigma^{-1}}\sim \sqrt{n}$, and (b) to expand $ \nabla f(\bar \theta_n)=\overline{\nabla f (\theta_n)}\sim \frac{\eta_n^{\text{sto}}-\eta_0}{\gamma n}+1/\sqrt{n}$. 

\vspace*{-.1444cm}

 \item 
 As for Proposition~\ref{prop:phifunction}, the constraint on the step-size $\gamma$ depends on the function $h$. Moreover the step-size $\gamma$ is constant, contrary to previous works on SDA \citep{xia10} which prove results for decreasing step-size $\gamma_n=C/\sqrt{n}$ for the convex case (and with a convergence rate of only $O(1/\sqrt{n})$).

\vspace*{-.1444cm}

 \item 
 The first term is the ``bias'' term. It only depends on the ``distance'' from the initial point $\theta_0$ to the solution $\theta_*$ as the minimum of two terms. The first one recovers the deterministic bound of Proposition~\ref{prop:phifunction}. The second one, specific to quadratic objectives, leads to an accelerated rate of $O(1/n^2)$ for some good starting points such that $\Vert \nabla h (\theta_0)-\nabla h (\theta_*)\Vert^2_{\Sigma^{-1}}<\infty$, thus extending the result from~\citet{FlaBac15}. 

 \vspace*{-.1444cm}

\item 
 The second term is the ``variance'' term which depends on the noise in the gradients. When the noise is structured (such as for least-squares regression), i.e, there exists $\sigma>0$ such that $C\preccurlyeq \sigma^2 \Sigma$, the variance term becomes $\frac{\sigma^2d}{n}$ which is optimal over all estimators in ${\RR^d}$  without regularization \citep{Tsyb}. However the regularization $g$ does not bring  statistical improvement as  possible, for instance, with $\ell_1$-regularization. We believe this is due to our proof technique. Indeed, in the case of linear constraints, \citet{DucRua16} recently showed that the primal iterates $(\theta_n)$ follow a central limit theorem (CLT), namely $\sqrt{n}\bar \theta_n$ is asymptotically normal with a covariance precisely restricted to the active constraints. This supports that SDA may leverage the regularization (the active constraints in their case) to get better statistical performance. We leave such non-asymptotic results to future work.
\end{itemize}

Assumption \textbf{(A6)} on the gradient noise is quite general, since the noise $(\xi_n)$ is allowed to be a martingale difference sequence (correct conditional expectation given the past, but not necessarily independence from the past). However it is not verified by the oracle corresponding to regular
SDA for least-squares regression, where the noise combines both an additive  and a multiplicative part, and its covariance is then no longer bounded in general (it will be for $g$ the indicator function of a bounded set).

\subsection{Least-squares regression}\label{sec:leastsquares}

We consider now the least-squares regression framework, i.e, risk minimization with the square loss. Following \citet{bm1}, we assume that:
\begin{description}

\vspace*{-.1444cm}

\item [\textbf{(A7)}] 
The observations $(x_n,y_n)\in{\RR^d}\times \RR$, $n\geq1$, are i.i.d.~distributed with finite variances $\EE \Vert x_n\Vert_2^2<\infty$ and $ \EE y_n^2<\infty$.

\vspace*{-.1444cm}

\item [\textbf{(A8)}] 
We consider the \emph{least-squares regression} problem which is the minimization of the quadratic function $f(\theta)=\frac{1}{2}\EE (\langle x_n,\theta\rangle -y_n)^2$. 

\vspace*{-.1444cm}

\item [\textbf{(A9)}] 
We denote by $\Sigma=\EE [x_n\otimes x_n]$ the population covariance matrix, which is the Hessian of $f$ at all points. Without loss of generality, we reduce ${\RR^d}$ to the minimal subspace where all $x_n$, $n\geq1$, lie almost surely. Therefore $\Sigma$ is invertible and all the eigenvalues of $\Sigma$ are strictly positive, even if they may be arbitrarily small.

\vspace*{-.1444cm}

\item  [\textbf{(A10)}]
We denote the residual by $\xi_n=(y_n-\langle\theta_*,x_n\rangle)x_n$. We have $\EE [\xi_n]=0$ but $\EE [\xi_n\vert x_n]\neq 0$ in general (unless the model is well-specified).
There exists $\sigma>0$ such that $\EE[\xi_n\otimes \xi_n] \preccurlyeq \sigma^2\Sigma$.

\vspace*{-.1444cm}

\item [\textbf{(A11)}]
There exists $\kappa>0$ such that for all $z\in{\RR^d}$, $\EE \langle z,x_n\rangle ^4\leq \kappa \langle z,\Sigma z\rangle$.

\vspace*{-.1444cm}

\item  [\textbf{(A12)}]
The function $g$ is lower bounded by some constant which is assumed by sake of simplicity to be $0$.

\vspace*{-.1444cm}

\item  [\textbf{(A13)}]
There exists $L>0$ such that $L h-\frac{1}{2}\Vert\cdot\Vert_\Sigma^2$ is convex.
\end{description}

Assumptions \textbf{(A7-9)} are standard for least-squares regression, while Assumption \textbf{(A10)}  defines a bounded statistical noise.
Assumption \textbf{(A11)} is commonly used in the analysis of least-mean-square algorithms \citep{Mac95} and says the projection of the covariates $x_n$ on any direction $z\in {\RR^d}$ have a bounded \emph{kurtosis}. It is true for Gaussian vectors with $\kappa = 3$. Assumption \textbf{(A13)} links up the geometry of the function $h$ and the objective function $f$; for example for $\ell_p$-geometries, $L$ is proportional to $\EE \| x\|_q^2$ where $1/p+1/q=1$ (see Corollary~\ref{cor:pdivergence} in \myapp{examples}). 

For the least-squares regression problem, the SDA algorithm defined in \eq{sda} takes the form:
\BEA\label{eq:daquad}
\eta_n
&=&
\eta_{n-1}-\gamma \big(\langle x_{n},\theta_{n-1}\rangle -y_n\big)x_{n}\nonumber \\
\theta_n
&=&
\nabla h_n^*(\eta_n). 
\EEA
This corresponds to a stochastic oracle of the form $\nabla f_n(\theta)= (\Sigma+\zeta_n)(\theta-\tnc)-\xi_n$ for $\theta\in{\RR^d}$, with $\zeta_n=x_n\otimes x_n-\Sigma$. This oracle combines an additive noise $\xi_n$ satisfying the previous Assumption~\textbf{(A6)} and a multiplicative noise $\zeta_n$ which is  harder to analyze.

We obtain a similar result compared to Proposition~\ref{prop:theta} at the cost of additional corrective terms. 
\begin{proposition}\label{prop:thetasto}
Assume  \textbf{(A2-4)} and \textbf{(A7-13)}. Consider the recursion in \eq{daquad} for any constant step-size $\gamma$  such that $\gamma \leq   \frac{1}{4\kappa L d}$. Then 
 \[
  \frac{1}{2}{\EE \Vert \bar \theta_{n}-\theta_*\Vert_\Sigma^2}
  \leq
    { 2\frac{D_h(\theta_*,\theta_{0})}{\gamma n}}
      +\frac{ 32 d }{n}\big(\sigma^2+\kappa  \Vert  \theta_{*}-\tnc\Vert_\Sigma^2\big)
+\frac{16\kappa d }{n^2}\bigg( \frac{5D_h(\theta_*,\theta_{0})}{\gamma }+g(\theta_0)\bigg).
   \]
\end{proposition}

We can make the following remarks:

\vspace*{-.1444cm}

\begin{itemize}
\item 
The proof technique is similar to the one of Proposition~\ref{prop:theta}. Nevertheless its complexity comes from the extra multiplicative noise $\zeta_n$ in the gradient estimate  (see \myapp{proofsto}). 

\vspace*{-.1444cm}

\item 
The result is only proven for $\gamma \leq 1/(4\kappa L d)$ which seems to be a proof artifact. Indeed we empirically observed (see \mysec{exp}) that the iterates still converge to the solution for all $\gamma \leqslant 1 / (2 \EE\Vert x_n\Vert_2^2)$. 

\vspace*{-.1444cm}

\item The global bound leads to a rate of $O(1/n)$ without strong convexity, which is optimal for stochastic approximation, even with strong convexity   \citep{NemYud83}. We recover the   terms of Proposition~\ref{prop:theta} pertubed by: (a) one corrective term of order $O(d/n)$ which depends on the distance between the solution $\theta_{*}$ and the global minimizer $\tnc$ of the quadratic function $f$, which corresponds to the covariance of the multiplicative noise at the optimum, and (b) two residual terms of order $O(d/n^{2})$. It would be interesting to study whether these two terms can be removed.

\vspace*{-.1444cm}

\item  As in Proposition~\ref{prop:theta}, the bias is also $O\big(\frac{1} {(\gamma n)^2}\Vert \nabla h (\theta_0)-\nabla h (\theta_*)\Vert^2_{\Sigma^{-1}}\big)$ for specific starting points (see proof in \myapp{proofsto} for details).

\vspace*{-.1444cm}

\item It is worth noting that in the constrained case ($g=\mathbbm{1}_{\mathcal{C}}$ for a bounded convex set $\mathcal{C}$), the covariance of the noisy oracle is simply bounded by  $(\kappa \tr \Sigma r^2+\sigma^2) \Sigma$ where we denote by $r=\max_{\theta\in\mathcal C}\Vert \theta-\tnc\Vert_2$ (see \myapp{stobounded} for details). Therefore Proposition~\ref{prop:theta} already implies  $\textstyle  \frac{1}{2}{\EE \Vert  \bar \theta_n - \theta_*\Vert_{\Sigma}^2}
  \leq
  2 \frac{D_h(\theta_{*},\theta_{0})}{\gamma n} +\frac{8d}{n} (\sigma^2 +\kappa r^2 \tr \Sigma)$. Moreover the result holds then for any step-size $\gamma \leqslant 1/L$, which is bigger  than allowed for $g=0$~\citep{bm1}.
\end{itemize}

\subsection{Convergence results on the objective function}
In this section we present the convergence properties of the SDA method on the  objective function $\psi=f+g$ rather than on the norm $\Vert\cdot\Vert_{\Sigma}$.

We first start with a disclaimer: it is not possible to obtain general non-asymptotic results on the convergence of the SDA iterates in term of   function values without additional assumptions on the regularization $g$. We indeed show in \myapp{lowerquad} that, even in the simple case of a linear function  $f(\theta)=\langle a,\theta\rangle$, for $a\in{\RR^d}$, we can always find, for any finite time horizon $N$, a quadratic non-strongly convex regularization function $g_N$ such that for any unstructured noise of variance $\sigma^2$, the function value $\psi_N(\theta)=f(\theta)+g_N(\theta)$ evaluated in the SDA iterates at time $N$ is lowerbounded by
\[
\psi_N(\bar \theta_{N})-\psi_N(\theta_{*})\geq \frac{\sigma^2}{12}.
\]
This lower bound is specific to the SDA algorithm and we underline that the regularization $g_N$ depends on the horizon $N$. However this result still prevents the possibility of a universal non-asymptotic convergence result on the function value for the SDA iterates for general quadratic and linear functions. We note that this does not apply  to the setting of Proposition~\ref{prop:theta} and Proposition~\ref{prop:thetasto} since $\Sigma=0$ for a linear function and the  vector $q$ defining the linear term $\langle q, \theta \rangle$ cannot be in the column space of $\Sigma$, thus violating Assumption \textbf{(A5)}.  We conjecture that in the  setting of Assumption~\textbf{(A5)}, the lower bound is $O(1/\sqrt{n})$ as well.

We  now provide some specific examples for which we can prove convergence in function values. 

\paragraph{Quadratic objectives with smooth regularization.}
When there exists a constant $L_g\geq 0$ such that $L_g f-g$  is convex on $\mathring{\mathcal X}$ then results from Propositions \ref{prop:theta} and  \ref{prop:thetasto} directly imply convergence of the composite objective to the optimum through
\[
 \psi(\bar \theta_n)-\psi(\theta_*)\leq \frac{(L_g+1)}{2}\Vert \bar \theta_n -\theta_*\Vert_\Sigma^2 = O(1/n),
\]
with precise constants from Propositions \ref{prop:theta} and  \ref{prop:thetasto}. Indeed we have in that case $(L_g+1) f-\psi$  convex and  this would be directly implied by Proposition~\ref{prop:lc} in \myapp{deterministic}. 

An easy but still interesting application is the non-regularized case ($g=0$) when the optimum~$\theta_*$ is the global optimum $\tnc$ of $f$,  because then $\psi(\theta)-\psi(\theta_*)=\frac{1}{2}\Vert \theta -\theta_*\Vert_\Sigma^2$. Thus this extends previous results on function values~\citep{DieFlaBac16} to non-Euclidean geometries. 

\paragraph{Constrained problems.}

When $g$ is the indicator function of a convex set $\mathcal{C}$ then by definition the primal iterate $\theta_n\in\mathcal{C}$ and by convexity  $\bar \theta_n\in\mathcal{C}$. Therefore  $\psi(\bar \theta_n)=f(\bar \theta_n)+\mathbbm{1}_{\mathcal{C}}(\bar \theta_n)=f(\bar \theta_n)$ and we obtain with the Cauchy-Schwarz inequality:
\BEAS
 f(\bar \theta_n)-f(\theta_*)
 &=&
 \langle \nabla f(\theta_*),\bar \theta_n-\theta_*\rangle+\frac{1}{2}\Vert \bar \theta_n-\theta_*\Vert_\Sigma^2\\[-.1cm]
 &\leq& \Vert \theta_*-\tnc\Vert_2 \Vert \bar \theta_n-\theta_*\Vert_\Sigma+\frac{1}{2}\Vert \bar \theta_n-\theta_*\Vert_\Sigma^2 = O\Big(\frac{\Vert \theta_*-\tnc\Vert_2 }{\sqrt{n}}\Big),
\EEAS
with precise constants from Propositions \ref{prop:theta} and  \ref{prop:thetasto}. Hence we obtain a global rate of order $O(1/\sqrt{n})$ for the convergence of the function value in the constrained case. 

These rates may be accelerated to $O(1/{n})$ for certain specific convex constraints or when the global optimum $\tnc\in\mathcal C$;  \citet{DucRua16} recently obtained asymptotic convergence results for the iterates  in the cases of linear and $\ell_2$-ball constraints for linear objective functions. Their results can be directly extended to asymptotic convergence of function values and very probably to all strongly convex sets \citep[see, e.g.,][]{Via83}. However, even for the simple $\ell_2$-ball constrained problem, we were not able to derive non-asymptotic convergence rates for   function values. 

However the global rate of order $O(1/\sqrt{n})$ is statically non-improvable in general. In \myapp{loweroracle}, we relate the stochastic convex optimization problem \citep{agarwal2010information} to the statistical problem of convex aggregation of estimators \citep{tsybakov2003optimal,Lec06}. These authors showed lower bounds on the performance of such estimators which provide us lower bounds on the performance of any stochastic algorithm to solve constrained problems.
In Proposition~\ref{prop:oraclequad} and Proposition~\ref{prop:oraclelin} of \myapp{loweroracle}, we derive more precisely lower bound results for linear and quadratic functions for certain ranges of $n$ and $d$ confirming the optimality of the convergence rate $O(1/\sqrt{n})$. This being said, in our experiments in \mysec{exp}, we observed that the convergence of function values follows closely the convergence in the Mahalanobis distance.

\section{Parallel between dual averaging and mirror descent}
\label{sec:mdda}

In this section we compare the behaviors of DA and MD algorithms, by highlighting their similarities and differences, in particular in terms of continuous-time interpretation.

\subsection{Lazy versus greedy projection methods}

DA and MD are  often described in the online-learning literature as ``lazy'' and ``greedy'' projection methods \citep{Zin03}. Indeed, the difference between these two methods  is more apparent in the Euclidean projection case (when $g=\mathbbm{1}_{\mathcal{C}}$ and $h=\frac{1}{2}\Vert \cdot \Vert_2^2$). MD is then projected gradient descent and may be written  under its primal-dual form as:
\[ \eta^{\text{md}}_n=\theta^{\text{md}}_{n-1}-g^{\text{md}}_{n}
\quad \text{with} \quad
g^{\text{md}}_{n}\in\partial f(\theta^{\text{md}}_{n-1}) 
\quad \text{ and } \quad 
\theta^{\text{md}}_n=\argmin_{\theta\in\mathcal{C}} \Vert \eta^{\text{md}}_n-\theta\Vert_2.
 \]
Whereas DA takes the form 
\[
 \eta^{\text{da}}_n=\eta^{\text{da}}_{n-1}-g^{\text{da}}_{n} 
 \quad \text{with} \quad
g^{\text{da}}_{n}\in\partial f(\theta^{\text{da}}_{n-1}) 
\quad \text{ and } \quad  \theta^{\text{da}}_n=\argmin_{\theta\in\mathcal{C}} \Vert \eta^{\text{da}}_n-\theta\Vert_2.
\]
Therefore, imagining the subgradients $g_n$ are provided by an adversary without the need to compute the primal sequence $(\theta_n)$, no projections are needed to update the dual sequence $(\eta^{\text{da}}_n)$, and this one moves far away in the asymptotic direction of the  gradient at the optimum $\nabla f(\theta_*)$. Furthermore the primal iterate  $\theta^{\text{da}}_n$ is simply obtained, when required, by projecting back the dual iterate in the constraint set. Conversely, the MD dual iterate $\eta^{\text{md}}_n$ update calls for $\theta^{\text{md}}_{n-1}$, and therefore a projection step is unavoidable. Thereby MD iterates $(\eta^{\text{md}}_n,\theta^{\text{md}}_n)$ are going, at each iteration,  back-and-forth between the boundary and the outside of the convex set $\mathcal{C}$.

\subsection{Strongly convex cases}
MD converges linearly  for  smooth and strongly convex functions $f$, in the absence of a regularization component  \citep{LuFreNes16} or for Euclidean geometries
\citep{Nes13}. However we were not able to derive faster convergence rates for DA when the function $f$ or the regularization $g$ are strongly convex. Moreover the only results we found in the literature are about (a) an alteration of the dual gradient method \citep[][Section 4]{dev} which is itself a modification of DA with an additional projection step proposed by \citet{Nes13} for smooth optimization, (b) the strongly convex regularization $g$ which enables \citet{xia10} to obtain a $O(1/\mu n)$ convergence rate in the stochastic case.

At the simplest level, for $h=\frac{1}{2}\Vert\cdot\Vert_{2}^{2}$  and  $f=0$,  MD is equivalent to the proximal point algorithm \citep{Mar70}
$
\theta^{\text{md}}_{n}=\argmin_{\theta \in {\RR^d}} \big\{ g(\theta) +\frac{1}{\gamma} \Vert \theta-\theta^{\text{md}}_{n-1}\Vert_{2}^{2}\big \},
$
whereas DA, which is not anymore iterative, is such that 
$
\theta^{\text{da}}_{n}=\argmin_{\theta \in {\RR^d}} \big\{ g(\theta) +\frac{1}{\gamma n } \Vert \theta \Vert_{2}^{2}\big \}.
$
For the squared $\ell_2$-regularization $g(\theta)=\frac{\nu}{2}\Vert \theta-\theta_*\Vert_2^2$, we compute exactly (see \myapp{compa}) 
\[
g(\theta^{\text{md}}_n)-g(\theta_*)=\Big(\frac{1}{\gamma \nu}\Big)^n[g(\theta^{\text{md}}_0)-g(\theta_*)] \quad \text{ and } \quad g(\theta^{\text{da}}_n)-g(\theta_*)=\frac{g(\theta_*)}{(\gamma n)^2}.
\]
Therefore the convergence of DA can be dramatically slower than MD. However when noise is present, its special structure may be leveraged to get interesting results. 

\subsection{Continuous time interpretation of DA et MD}\label{sec:continuousint}

Following \citet{NemYud83,Kri15,WibWilJor16} we propose a continuous interpretation of these methods  for $g$ twice differentiable. Precise computations are derived in \myapp{continuousint}. 

The MD iteration in  \eq{mdprox}  may be viewed as a forward-backward Euler discretization of the MD ODE  
\begin{equation}\label{eq:odemd1}
\dot \theta = -\nabla^2h(\theta)^{-1}[\nabla f(\theta)+\nabla g(\theta)].
\end{equation}
On the other hand, the ODE associated to DA takes the form
\begin{equation}\label{eq:odeda1}
 \dot \theta= -\nabla^2(h(\theta)+t g(\theta))^{-1}(\nabla f(\theta)+\nabla g(\theta)).
\end{equation}
It is worth noting that these ODEs are very similar, with an additional term $t g(\theta)$ in the inverse mapping $\nabla^2(h(\theta)+t g(\theta))^{-1}$ which may slow down the DA dynamics.

In analogy with the discrete case, the Bregman divergences $D_h$ and $D_{h+tg}$ are respectively   Lyapunov functions for the MD and the DA ODEs \citep[see, e.g.,][]{Kri15} and we notice  in \myapp{continuousint} the continuous time argument  really mimics the proof of Proposition~\ref{prop:phifunction} without the technicalities associated with    discrete time.  Moreover we recover the variational interpretation of \citet{Kri15,WibWilJor16,WilRecJor16}: the Lyapunov function generates the dynamic in the sense that a function $L$ is first chosen and secondly a dynamics, for which~$L$ is a Lyapunov function, is then designed.  In this way MD and DA are the two different dynamics associated to the two different Lyapunov functions $D_h$ and $D_{h+tg}$. We also provide in \myapp{continuousint} a slight extension to the noisy-gradient case.

\section{Experiments}\label{sec:exp}
\label{sec:simulations}

In this section, we illustrate our theoretical results on synthetic examples. We   provide additional experiments on a standard machine learning benchmark in \myapp{expsido}.

\paragraph{Simplex-constrained  least-squares regression with synthetic data.}
We consider normally distributed inputs $x_n \in \RR^d$ with a covariance matrix $\Sigma$ that has random eigenvectors and eigenvalues $1/k$, for $k=1,\dots,d$ and a random global optimum $\theta_\Sigma\in[0,+\infty)^d$. The outputs $y_n$ are generated from a linear function with homoscedatic noise with unit signal to noise-ratio ($\sigma^2=1$). We denote by $R^2=\tr \Sigma$ the average radius of the data and we show results averaged over 10 replications. 

We consider the problem of least-squares regression constrained on the simplex $\Delta_d$ of radius $r=\Vert\theta_\Sigma\Vert_1/2$ , i.e.,  $\min_{\theta\in r\Delta_d}\EE (\langle x_n,\theta\rangle-y_n)^2$, for $d=100$.
We compare the performance of SDA and SGD algorithms with different settings of the step-size $\gamma_n$, constant or proportional to $1/\sqrt{n}$.  In the left plot of \myfig{synthetic} we show the performance on the objective function and on the right plot, we show the performance on the   squared Mahalanobis norm $\Vert\cdot\Vert_\Sigma^2$. All costs are shown in log-scale, normalized so that the first iteration leads to $f(\theta_0)-f(\theta_*)=1$. We can make the following observations (we only show results on Euclidean geometry since results under the negative entropy geometry were very similar):

\vspace*{-.1444cm}

\BIT
\item

 With constant step-size, SDA converges to the solution at rate $O(1/n)$ whereas the SGD algorithm  does not converge to the optimal solution.

\vspace*{-.1444cm}

 \item
 With decaying step-size $\gamma_n=1/(2R^2\sqrt{n})$, SDA and SGD converge first at rate $O(1/\sqrt{n})$, then at rate  $O(1/n)$, taking  finally advantage of the strong-convexity of the problem.

\vspace*{-.1444cm}

 \item
 We note (a) there is no empirical difference between the performance on the objective function and the  squared distance $\Vert\cdot\Vert_\Sigma^2$, (b) with decreasing step-size, SGD and SDA behave very similarly. 
 \EIT

 \begin{figure}[!h]
\centering
\begin{minipage}[c]{.45\linewidth}
\includegraphics[width=\linewidth]{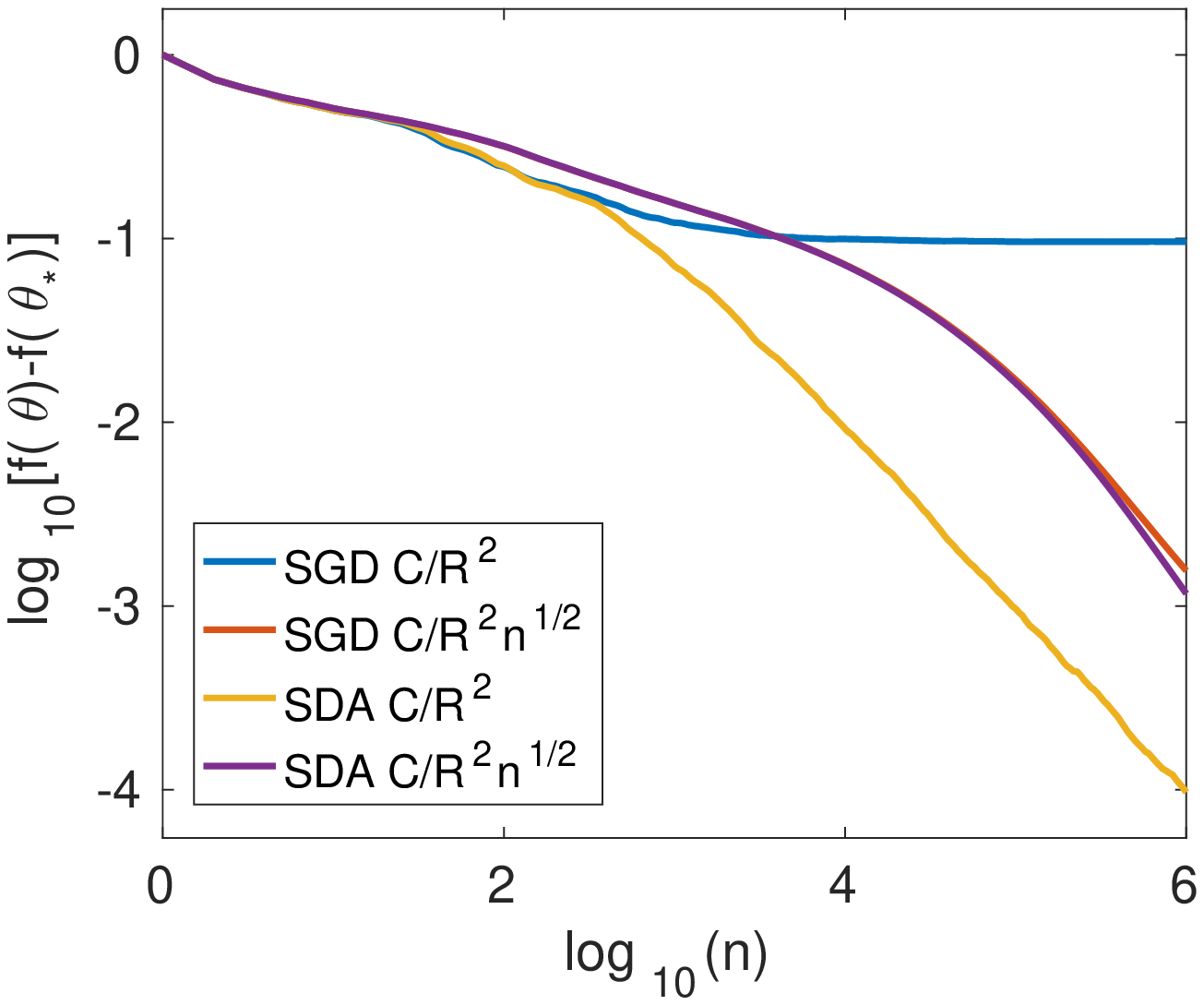}
   \end{minipage} \hspace*{.08\linewidth}
   \begin{minipage}[c]{.45\linewidth}
\includegraphics[width=\linewidth]{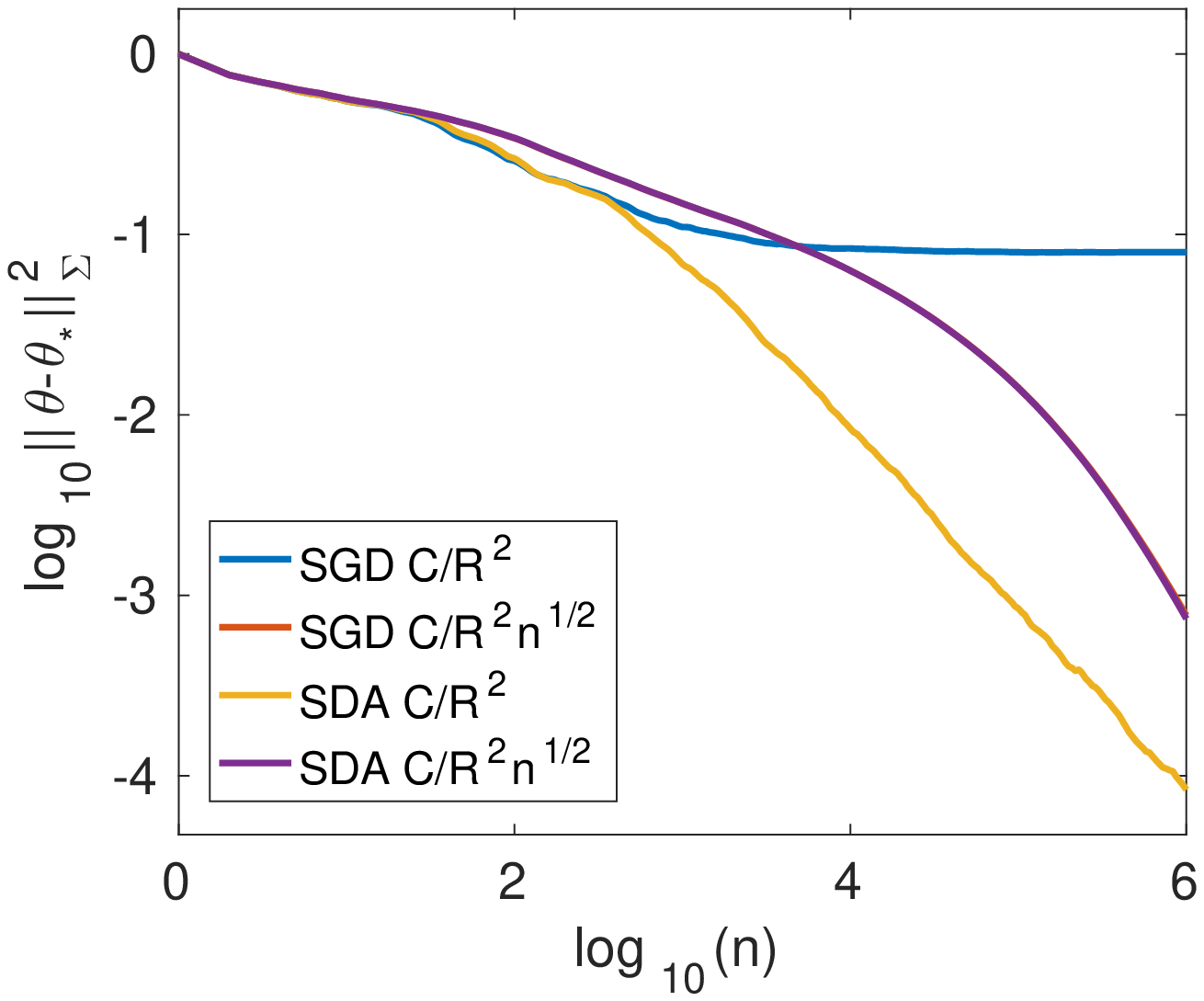}
   \end{minipage} 
  \caption{Simplex-constrained least-squares regression with synthetic data. Left: Performance on the objective function. Right: Performance on the Mahalanobis norm $\Vert\cdot\Vert_\Sigma^2$.}
     \label{fig:synthetic}
\end{figure}

\section{Conclusion}

In this paper, we proposed and analyzed the first algorithm to achieve a convergence rate of $O(1/n)$ for stochastic composite objectives, without the need for strong convexity. This was achieved by considering a constant step-size and averaging of the primal iterates in the dual averaging method.

Our results only apply to expectations of quadratic functions (but to any additional potentially non-smooth terms). In fact, constant step-size stochastic dual averaging is not convergent for general smooth objectives; however, as done in the non-composite case by~\citet{bm1}, one could iteratively solved quadratic approximations of the smooth problems with the algorithm we proposed in this paper to achieve the same rate of $O(1/n)$, still with robustness to ill-conditioning and efficient iterations. 
Finally, it would be worth considering accelerated extensions to achieve a forgetting of initial conditions in $O(1/n^2)$.

\subsection*{Acknowledgements}
 The authors would like to thank Aymeric Dieuleveut and  Damien Garreau for helpful discussions.
 \bibliographystyle{plainnat}
\bibliography{colt2017}

\newpage

\appendix

\section{Unambiguity of the primal iterate}\label{app:dadefined}
We describe here conditions under which the primal iterate $\theta_n$ in \eq{da} is correctly defined. Since $h$ is strictly convex, $h_n^*$ is continuously differentiable on $ \mathring \dom  h_n^*$ \citep[see][Theorem 4.1.1]{marechal}. Therefore the primal iterate $\theta_n$ is well defined if the dual iterate $\eta_n\in \mathring\dom  h_n^*$. It is, for example, the case under two natural assumptions as shown by the next lemma which is an adaption of Lemma 2 by \citet{BauBolTeb15}.
\begin{lemma} \label{lemma:dadefined}
We make the following assumptions:
 \begin{description}
  \item[\textbf{(B1)}]   $h$ or $g$ is supercoercive.
  \item[\textbf{(B2)}]   $\argmin_{\theta\in\mathcal{X}} \psi(\theta)$ is compact and $h$ bounded below.
 \end{description}
 Under \textbf{(B1)} or \textbf{(B2)} the primal iterates $(\theta_n)$ defined in \eq{da} are well defined.
\end{lemma}

\begin{proof}
Since $h$ is strictly convex, $h_n^*$ is continuously differentiable on $ \mathring \dom  h_n^*$ \citep[see][Theorem 4.1.1]{marechal}. Therefore the primal iterate $\theta_n$ is well defined if the dual iterate $\eta_n\in \mathring\dom  h_n^*$.
\begin{itemize}                                                                                                                                                                                                                                                                                                       \item If $h$ or $g$ is supercoercive then $h_n$ is supercoercive \cite[see][Proposition 11.13]{comb} and it follows from \citet[Chapter E, Proposition 1.3.8]{marechal} that $\dom h_n^* ={\RR^d}$.
\item  If $\argmin_{\theta\in\mathcal{X}}\{ \psi(\theta)\}$ is compact then  $\psi+\mathbbm{1}_\mathcal{X}$ is coercive. Moreover
\BEAS
 h_n^*(\eta_n)
 &=&
 \sup_{\theta\in\mathcal{X}}\big\{ \langle \eta_n,\theta\rangle -h_n(\theta) \big\} \text{ since} \mathcal{X}\subset \overline{ \dom }h \\
 &=&
 -\inf_{\theta\in\mathcal{X}} \Big\{ h(\theta)+\gamma \sum_{i=1}^n(g(\theta)+f(\theta_{i-1})+\langle \nabla f(\theta_{i-1}), \theta-\theta_{i-1}\rangle) \Big\}\\
 &&+\gamma \sum_{i=1}^n(f(\theta_{i-1})-\langle \nabla f(\theta_{i-1}),\theta_{i-1}\rangle)\\
 &\leq&
 -\inf_{\theta\in\mathcal{X}} \big\{ h(\theta)+ n \gamma (g(\theta)+f(\theta)) \big\} \text{ by convexity of } f\\
  &&+\gamma \sum_{i=1}^n(f(\theta_{i-1})-\langle \nabla f(\theta_{i-1}),\theta_{i-1}\rangle).
 \EEAS
Therefore $\eta_n\in \mathring\dom  h_n^*$ since $\psi+\mathbbm{1}_\mathcal{X}$ is coercive and $h$ bounded below \citep[see][Proposition 11.15]{comb}.                                                                                                                                                                                                                                                                                                     \end{itemize}
\end{proof}

\section{Proof of convergence of deterministic DA}\label{app:deterministic}

We first describe a new notion of smoothness defined by \citet{BauBolTeb15}. Then we present our extension of the Bregman divergence to the non-smooth function $g$ to finally prove Proposition~\ref{prop:phifunction}.

\subsection{A  \emph{Lipschitz-like/convexity condition}}

Classical results on the convergence of optimization algorithms in non-Euclidean geometry assume on one hand that the function $h$ is strongly convex and on the other hand the function $f$ is Lipschitz or smooth. 
Following \citet{BauBolTeb15,LuFreNes16}, we consider a different assumption which combines the smoothness  of $f$ and the strong convexity of $h$  on a single condition called \emph{Lipschitz-like/Convexity Condition} by \citet{BauBolTeb15} and denoted by \textbf{(LC)}:
\begin{description}
 \item [\textbf{(LC)}] There exists a constant $L\in\RR$ such that $Lh-f$ is convex on $\mathring  {\mathcal{X}}$.
\end{description}
For Euclidean geometry, this condition is obviously equivalent to the smoothness of the function $f$ with regards to the $\ell_2$-norm. Moreover, under an arbitrary norm $\Vert \cdot \Vert$, assuming $h$ $\mu$-strongly convex and $f$ $L$-smooth  clearly implies, by simple convex computation, \textbf{(LC)} with constant $L/\mu$. 
However \textbf{(LC)} is much more general and may hold even when $f$ is non-smooth what precisely justifies the introduction of this condition. Many examples are described by \citet{BauBolTeb15,LuFreNes16}. Furthermore this notion has the elegance of pairing well with Bregman divergences and leading to more refined proofs as shown in the following proposition which summarizes equivalent properties of \textbf{(LC)}. 
\begin{proposition}[\citet{BauBolTeb15}]\label{prop:lc}
Assume \textbf{(A1-4)}. For $L>0$ the following conditions are equivalent:
\begin{itemize}
 \item $Lh-f$ is convex on $\mathring{\mathcal X}$, i.e.,  \textbf{(LC)} holds,
 \item $D_f(\alpha,\beta)\leq L D_h(\alpha,\beta) $ for all $(\alpha,\beta)\in {\mathcal X}\times \mathring{\mathcal X}$.
\end{itemize}
Furthermore, when $f$ and $h$ are assumed twice differentiable, then the above is equivalent to 
\[
\nabla^2 f(\theta)  \preccurlyeq L \nabla^2 h(\theta)\quad \text{  for all }  \theta\in\mathring{\mathcal X}.
\]
\end{proposition}

\subsection{Generalized Bregman divergence}\label{app:nonsmoothbreg}

The Bregman divergence  was defined by \citet{Bre67} for a differentiable convex function $h$ as
\begin{equation}
 D_h(\alpha,\beta)=h(\alpha)-h(\beta)-\langle \nabla h(\beta),\alpha-\beta), \text{ for }(\alpha,\beta)\in \dom h \times \mathring \dom  h.
\end{equation}
It behaves as a squared distance depending on the function $h$ and  extends the computational properties of the squared $\ell_{2}$-norm to non-Euclidean spaces. Indeed most proofs in Euclidean space rest upon the expansion $\Vert \theta_{n}-\theta_*-\gamma\nabla f(\theta_n)\Vert_{2}^{2}=\Vert \theta_{n}-\theta_*\Vert_{2}^{2}+\gamma^2\Vert \nabla f(\theta_n)\Vert_{2}^{2}-2 \gamma \langle \nabla f(\theta_n),\theta_{n}-\theta_*\rangle$ which is not available in non-Euclidean geometry. Therefore the Bregman divergence comes to rescue and is used to compute a deviation between the current iterate of the algorithm and the solution of the problem and, seemingly, used as an non-Euclidean Lyapunov function. 
It has been widely used in optimization \citep[see, e.g.,][for a review]{BauBor97}. 

We follow this path and include the regularization component $g$ of the objective function $\psi=f+g$ in the Bregman divergence for the sake of the analysis. If $g$ was differentiable we would simply use $D_{h+n\gamma g}$. Since $g$ is not differentiable, $D_{h_n}$ is not well defined. However for $(\alpha,\eta)\in \dom h \times \mathring \dom  h_n^*$, we denote by extension for $\theta=\nabla h_n^*(\eta)$:
\begin{equation}\label{eq:bregetend}
 \tilde D_n(\alpha,\eta)= h_n(\alpha)-h_n(\theta)-\langle \eta,\alpha-\theta\rangle.
\end{equation}
This extension is different from the one defined by \citet{Kiw97}.
It is worth noting that if there exists $\mu$ such that $\alpha=\nabla h_n^*(\mu)$, we recover the classical formula $\tilde D_n(\alpha,\eta)=D_{h_n^*}(\eta,\mu)$ which is well defined since $h_n^*$ is differentiable. Yet $\tilde D_n$ is defined more generally since such a $\mu$ does not always exist. The next lemma relates $\tilde D_{n}$ to $D_{h}$ and is obvious if $g$ is differentiable since $D_{h_n}=D_h+\gamma n D_g$.

\begin{lemma}\label{lem:Breg}
Let $n\geq0$, $\alpha\in \dom h$ and $\eta\in\mathring \dom h_n^*$, then with $\theta=\nabla h_n^*(\eta)$,
 \begin{equation}
  \tilde D_n(\alpha,\eta)\geq D_h(\alpha,\theta).
 \end{equation}
\end{lemma}
\begin{proof}
$\theta=\nabla h_{n}^*(\eta)$, thus $\eta\in \partial h_n(\theta)$ and by elementary calculus rule  $\partial h_n(\theta)=\nabla h (\theta)+n\gamma \partial g(\theta)$. Consequently $\eta-\nabla h (\theta)\in n\gamma \partial g(\theta)$ and by convexity of $g$
\[
 \tilde D_{n}(\alpha,\eta)-D_h(\alpha,\theta)=n\gamma \bigg[g(\alpha)-g(\theta)-\bigg\langle \frac{\eta-\nabla h (\theta)}{\gamma n}, \alpha-\theta\bigg\rangle\bigg]\geq 0.
\]
\end{proof}
\subsection{Proof of Proposition~\ref{prop:phifunction}}
We assume their exists a constant $L>0$ such that $Lh-f$ is convex on $\mathring {\mathcal X}$ and we assume the step-size $\gamma\leq 1/L$.
We first show that the Bregman divergence decreases along the iterates \citep[see, e.g.,][]{Teb03,bachduality}. For all $\theta\in \mathcal X$,
 \BEAS
\tilde D_n(\theta,\eta_n)-\tilde D_{n-1}(\theta,\eta_{n-1})&=&h_{n-1}(\theta_{n-1})-h_n(\theta_{n})+h_{n}(\theta)-h_{n-1}(\theta)\\
&&-\langle \eta_n,\theta-\theta_n \rangle+\langle \eta_{n-1} , \theta-\theta_{n-1}\rangle \\
&=&h_{n-1}(\theta_{n-1})-h_{n-1}(\theta_{n})-\gamma (g(\theta_n)-g(\theta))\\
&&+\langle \eta_{n-1},\theta_n-\theta_{n-1}\rangle+\langle \eta_n-\eta_{n-1},\theta_n-\theta \rangle\\
&=&-\tilde D_{n-1}(\theta_n,\eta_{n-1})-\gamma (g(\theta_n)-g(\theta))-\gamma\langle \nabla f(\theta_{n-1}),\theta_n-\theta\rangle.
\EEAS
Therefore  for all $\theta\in \mathcal X$,
\begin{multline}\label{eq:Breglya}
 \tilde D_n(\theta,\eta_n)-\tilde D_{n-1}(\theta,\eta_{n-1})=-\tilde D_{n-1}(\theta_n,\eta_{n-1}) +\gamma \langle \nabla f(\theta_{n-1}), \theta_{n-1}-\theta_n \rangle
 \\
- \gamma (g(\theta_n)-g(\theta)) -\gamma \langle \nabla f(\theta_{n-1}), \theta_{n-1}-\theta \rangle.
\end{multline}
It follows from Proposition~\ref{prop:lc} and Lemma~\ref{lem:Breg}
\[
f(\theta_n)-f(\theta_{n-1})+\langle \nabla f(\theta_{n-1}),\theta_n-\theta_{n-1}\rangle\leq L D_h(\theta_n,\theta_{n-1})\leq  L D_{n-1}(\theta_n,\theta_{n-1}),
\]
and from the convexity of $f$, 
\[
- \langle \nabla f(\theta_{n-1}), \theta_{n-1}-\theta \rangle\leq f(\theta)-f(\theta_{n-1}).
\]
And \eq{Breglya} is bounded by
\[
 \tilde D_n(\theta,\eta_n)-\tilde D_{n-1}(\theta,\eta_{n-1})\leq\gamma ( \psi(\theta)- \psi(\theta_n))+(\gamma L-1) D_h(\theta_n,\theta_{n-1}).
\]
Thus for $\gamma\leq 1/L$, 
\[
 \tilde D_n(\theta,\eta_n)-\tilde D_{n-1}(\theta,\eta_{n-1})\leq\gamma ( \psi(\theta)- \psi(\theta_n)).
\]
Taking $\theta=\theta_{n-1}$  we note that the sequence $\{\psi(\theta_n)\}_{n\geq 0}$ is decreasing and we obtain for $\gamma\leq 1/L$,
\begin{equation}
  \psi(\theta_n)-\psi(\theta)\leq \frac{1}{n+1}\sum_{k=0}^n [\psi(\theta_i)-\psi(\theta)]\leq \frac{D_h(\theta,\theta_0)-\tilde D_{n}(\theta,\eta_{n})}{\gamma (n+1)}.
\end{equation}

We assume now that the non-smooth part $g=0$ and there exists $\mu\geq0$ such that $ f-\mu h$  is convex.  So Proposition~\ref{prop:lc}  implies
\[
- \langle \nabla f(\theta_{n-1}), \theta_{n-1}-\theta \rangle\leq f(\theta)-f(\theta_{n-1})-\mu D_h(\theta,\theta_{n-1}),
\]  
which gives with \eq{Breglya} the better bound
\[
  D_h(\theta,\theta_n)-D_h(\theta,\theta_{n-1}) \leq \gamma ( f(\theta)- f(\theta_n))-\gamma \mu D_h(\theta,\theta_{n-1}) +(\gamma L-1) D_h(\theta_n,\theta_{n-1}).
\]
And  for $\gamma\leq 1/L$, this can be simplified as
\[
 D_h(\theta,\theta_n) \leq (1-\gamma \mu )D_h(\theta,\theta_{n-1}) +\gamma ( f(\theta)- f(\theta_n)).
 \]
 The sequence $\{f(\theta_n)\}_{n\geq 0}$ is still decreasing and we obtain by  expanding the recursion
\BEAS
 D_h(\theta,\theta_n)&\leq& (1-\gamma \mu)^nD_h(\theta,\theta_0) +\sum_{k=1}^n (1-\gamma \mu) ^{n-k} \gamma (f(\theta)-f(\theta_k))\\
&\leq& (1-\gamma \mu)^n D_h(\theta,\theta_0) +
\sum_{k=1}^n (1-\gamma \mu) ^{n-k} \gamma (f(\theta)-f(\theta_k))\\
&\leq&(1-\gamma \mu)^nD_h(\theta,\theta_0) +
\sum_{k=1}^n (1-\gamma \mu) ^{n-k} \gamma (f(\theta)-f(\theta_n))\\
&\leq&(1-\gamma \mu)^nD_h(\theta,\theta_0) +
 \gamma \frac{1-(1-\gamma \mu)^n}{\gamma \mu} (f(\theta)-f(\theta_n)).
\EEAS
Thus for all $\theta\in \mathcal X$,
\[
  \frac{1-(1-\gamma \mu)^n}{ \mu} (f(\theta_n)-f(\theta))+ D_h(\theta,\theta_n)\leq (1-\gamma\mu)^nD_h(\theta,\theta_0),
\]
and 
\[
f(\theta_n)-f(\theta)\leq   \frac{ \gamma \mu (1-\gamma\mu)^n}{1-(1-\gamma \mu)^n} \frac	{D_h(\theta,\theta_0) }{\gamma}\leq (1-\gamma\mu)^n\frac{D_h(\theta,\theta_0) }{\gamma},
\]
since $(1-\gamma \mu)^2\leq 1-\gamma \mu$ implies $\gamma \mu/(1-(1-\gamma \mu)^n)\leq1$.

\section{Proof of Proposition~\ref{prop:theta}}\label{app:semisto}

In this section, we will prove Proposition~\ref{prop:theta}. The proof relies on considering the difference between the iteration with noise  we denote by $(\eta_n,\theta_n)$ and without noise we denote by $(\omega_n,\phi_n)$,  which happens to verify a similar recursion as the SDA recursion.
\begin{itemize}
 \item We first show in Lemma~\ref{lemma:etasemisto} that the distance  $\EE \Vert \eta_n-\omega_n\Vert_{\Sigma^{-1}}^2$ is of order $n$. 
 \item Then in Lemma~\ref{lemma:alphabeta} we show that $\EE \Vert \bar \theta_n-\bar \phi_n\Vert_\Sigma^2$ is of order $O(1/n)$, by: (a) noticing that $\EE \Vert\bar \theta_n-\bar \phi_n\Vert_\Sigma^2$ is of order $\frac{\EE \Vert \eta_n-\omega_n\Vert_{\Sigma^{-1}}^2}{n^2}+\frac{\text{variance}}{n}$, (b) combining this with the result of Lemma~\ref{lemma:etasemisto}.
\end{itemize}

\subsection{Two technical lemmas}

We first present and prove two technical lemmas.

\subsubsection{Bound on the difference of two dual iterates }
In the following lemma we show that the difference between two dual iterates  that follow the same recursion is of order $n$. This will be used with the iteration with noise $(\eta_n,\theta_n)$ and without noise $(\omega_n,\phi_n)$. 
\begin{lemma}\label{lemma:etasemisto}
Let us consider two sequences of iterates $(\mu_k,\alpha_k)$ and $(\nu_k,\beta_k)$ which satisfy the recursion  $\mu_n-\nu_n=\mu_{n-1}-\nu_{n-1}-\gamma \Sigma (\alpha_{n-1}-\beta_{n-1}) +\gamma\xi_n$, $\alpha_n=\nabla h_n^*(\mu_n)$ and $\beta_n=\nabla h_n^*(\nu_n)$  and assume that $\gamma$ is such that $2h-\gamma f $ is convex then for all $n\geq0$
\[
 \EE \Vert \mu_n-\nu_n\Vert_{\Sigma^{-1}}^2\leq  \Vert \mu_0-\nu_0 \Vert_{\Sigma^{-1}}^2   +n\gamma^2 \Tr \Sigma^{-1} C. 
\]
\end{lemma}
\begin{proof}
We first expand the square.
\BEAS
\Vert \mu_{n+1}-\nu_{n+1}\Vert_{\Sigma^{-1}}^2
&=&
\Vert \mu_{n}-\nu_{n}\Vert_{\Sigma^{-1}}^2 
+ \gamma^2 \Vert[ \Sigma  (\alpha_{n}-\beta_{n})-\xi_{n+1}] \Vert_{\Sigma^{-1}}^2 \\
&&
-2 \gamma \langle  \Sigma  (\alpha_{n}-\beta_{n})-\xi_{n+1} ,  \Sigma^{-1}(\mu_n-\nu_n)\rangle\\
&=&  
\Vert \mu_{n}-\nu_{n}\Vert_{\Sigma^{-1}}^2
+ \gamma^2 \Vert \alpha_{n}-\beta_{n} \Vert_\Sigma^2
+ \gamma^2 \Vert \xi_{n+1} \Vert_{\Sigma^{-1}}^2\\
&&
-2\gamma^2\langle  \alpha_{n}-\beta_{n},\xi_{n+1} \rangle 
-2 \gamma \langle   \alpha_{n}-\beta_{n}-\Sigma^{-1}\xi_{n+1} ,  \mu_n-\nu_n\rangle.
\EEAS
And taking the expectation 
\BEAS
 \EE[\Vert \mu_{n+1}-\nu_{n+1}\Vert_{\Sigma^{-1}}^2\vert \mathcal{F}_n]
 &=&
 \Vert \mu_{n}-\nu_{n}\Vert_{\Sigma^{-1}}^2 
 +\gamma^2 \EE [ \Vert \xi_{n+1} \Vert_{\Sigma^{-1}}^2\vert \mathcal{F}_n]\\
 &&
 + \gamma^2 \Vert \alpha_{n}-\beta_{n} \Vert_\Sigma^2
 -2\gamma^2\EE [\langle  \alpha_{n}-\beta_{n},\xi_{n+1} \rangle \vert \mathcal{F}_n]  \\
 &&
 -2 \gamma \EE [\langle   \alpha_{n}-\beta_{n})-\Sigma^{-1}\xi_{n+1} ,  \mu_n-\nu_n\rangle\vert \mathcal{F}_n] \\
 &=&
 \Vert \mu_{n}-\nu_{n}\Vert_{\Sigma^{-1}}^2 
 +\gamma^2  \tr \Sigma^{-1} \EE [\xi_{n+1} \otimes\xi_{n+1} \vert \mathcal{F}_n]\\
 &&
  + \gamma^2 \Vert \alpha_{n}-\beta_{n} \Vert_\Sigma^2
 -2\gamma^2\langle  \alpha_{n}-\beta_{n},\EE [\xi_{n+1}\vert \mathcal{F}_n] \rangle   \\
 &&
 -2 \gamma \langle   \alpha_{n}-\beta_{n})-\Sigma^{-1}\EE [\xi_{n+1}\vert \mathcal{F}_n] ,  \mu_n-\nu_n\rangle \\
  &=&
 \Vert \mu_{n}-\nu_{n}\Vert_{\Sigma^{-1}}^2 
   + \gamma^2  \tr \Sigma^{-1} C \\
 &&  + \gamma^2 \Vert \alpha_{n}-\beta_{n} \Vert_\Sigma^2
 -2 \gamma \langle   \alpha_{n}-\beta_{n} ,\mu_n-\nu_n\rangle.
\EEAS
Moreover, using the definition of $\alpha_n$ and $\beta_n$,
\BEAS
 \gamma \Vert \alpha_{n}-\beta_{n} \Vert_\Sigma^2
 -2 \langle   \alpha_{n}-\beta_{n} ,\mu_n-\nu_n\rangle
&=&
\langle \gamma \Sigma (\alpha_{n}-\beta_{n}) - 2(\mu_n-\nu_n), \alpha_{n}-\beta_{n}\rangle\\
&=&
\langle \gamma \nabla f(\alpha_n)-\nabla f(\beta_n)-2(\nabla h(\alpha_n)-\nabla h(\beta_n)), \alpha_{n}-\beta_{n}\rangle\\ 
&&- 2\langle (\mu_n-\nabla h(\alpha_n))-(\nu_n-\nabla h(\beta_n)), \alpha_{n}-\beta_{n}\rangle\\
&=& 
\langle \nabla(\gamma f-2h)(\alpha_{n})-\nabla(\gamma f-2h)(\beta_{n}), \alpha_{n}-\beta_{n}\rangle\\
&&- 2\langle (\mu_n-\nabla h(\alpha_n))-(\nu_n-\nabla h(\beta_n)), \alpha_{n}-\beta_{n}\rangle.
\EEAS
Using the $h$-smoothness of $f$ and assuming that $\gamma$ is such $2h-\gamma f$ is convex, 
\[
\langle \nabla(\gamma f-2h)(\alpha_{n})-\nabla(\gamma f-2h)(\beta_{n}), \alpha_{n}-\beta_{n}\rangle\leq 0, \]
and as explained in the proof of Lemma~\ref{lem:Breg}, $\mu_n-\nabla h(\alpha_n)\in\partial n\gamma g (\alpha_n)$ and $\nu_n-\nabla h(\beta_n)\in\partial n\gamma g (\beta_n)$ and consequently 
\[
\langle (\mu_n-\nabla h(\alpha_n))-(\nu_n-\nabla h(\beta_n)), \alpha_{n}-\beta_{n}\rangle\leq0,
\]
 by convexity of $g$.
This explains that 
\[
  \gamma \Vert \alpha_{n}-\beta_{n} \Vert_\Sigma^2
 -2 \langle   \alpha_{n}-\beta_{n} ,\mu_n-\nu_n\rangle\leq0.
\]
Then, taking the global expectation,  we have shown that 
\[
\EE \Vert \mu_{n+1}-\nu_{n+1}\Vert_{\Sigma^{-1}}^2
\leq 
\EE\Vert \mu_{n}-\nu_{n}\Vert_{\Sigma^{-1}}^2
+\gamma^2\tr \Sigma^{-1} C,
\]
which concludes the proof.
\end{proof}

\subsubsection{Bound on the difference of the average of two primal iterates}
In the following lemma we adapt the classic proof of averaged SGD by \citet{pj} to show that the difference between  two  averaged primal iterates, which follow the same recursion, is of order $O(1/n)$.
\begin{lemma}\label{lemma:alphabeta}
Let us consider two sequences of iterates $(\mu_k,\alpha_k)$ and $(\nu_k,\beta_k)$ which satisfy the recursion  $\mu_n-\nu_n=\mu_{n-1}-\nu_{n-1}-\gamma \Sigma (\alpha_{n-1}-\beta_{n-1}) +\gamma\xi_n$, $\alpha_n=\nabla h_n^*(\mu_n)$ and $\beta_n=\nabla h_n^*(\nu_n)$  and assume that $\gamma$ is such that $2h-\gamma f $ is convex then for all $n\geq0$
 \[
{\EE \Vert  \bar \alpha_n - \bar\beta_n\Vert_{\Sigma}^2}
\leq 
4 \frac{ \Vert \mu_0-\nu_0\Vert^2_{\Sigma^{-1}}}{(\gamma n)^2}  
+ {\frac{4}{n}\Tr \Sigma^{-1} C }.
\]
\end{lemma}

\begin{proof}
Let us consider two sequences of iterates $(\mu_k,\alpha_k)$ and $(\nu_k,\beta_k)$ which satisfy the recursion  $\mu_n-\nu_n=\mu_{n-1}-\nu_{n-1}-\gamma \Sigma (\alpha_{n-1}-\beta_{n-1}) +\gamma\xi_n$, $\alpha_n=\nabla h_n^*(\mu_n)$ and $\beta_n=\nabla h_n^*(\nu_n)$. This can be written as
\[
\Sigma(\alpha_n-\beta_n)
=\frac{\mu_n-\nu_n-\mu_{n+1}+\nu_{n+1}}{\gamma}+\xi_{n+1}.
\]
Thus we obtain
\BEAS
\Sigma^{1/2} \sum_{i=0}^{n-1}(\alpha_i-\beta_i)
&=&
\frac{\Sigma^{-1/2} (\mu_0-\nu_0-\mu_{n}+\nu_n)}{\gamma}
+\sum_{i=0}^{n-1}\Sigma^{-1/2} \xi_{i+1}.
\EEAS
Finally, using that by convexity $(a+b)^2\leq2(a^2+b^2)$, this leads to
\[
{\EE\Vert \bar \alpha_n- \bar \beta_n  \Vert_\Sigma^2}
\leq
2 \EE \Big\Vert \frac{\Sigma^{-1/2} (\mu_0-\nu_0)}{\gamma}+\sum_{i=0}^{n-1}\Sigma^{-1/2} \xi_{i+1}\Big\Vert_2^2+2\EE\Big\Vert \frac{\Sigma^{-1/2} (\mu_{n}-\nu_n)}{\gamma}  \Big\Vert_2^2.
\]
Using martingale second moment expansions, we obtain
\[
{\EE\Vert \bar \alpha_n- \bar \beta_n  \Vert_\Sigma^2}\leq 
   2\EE \frac{\Vert\mu_0-\nu_0\Vert^2_{\Sigma^{-1}}}{(\gamma n)^2}
    +2\frac{{\EE \Vert \mu_{n}-\nu_n\Vert^2_{\Sigma^{-1}}}}{(\gamma n)^2}  
  +\frac{2}{n^2}{\sum_{i=0}^{n-1}\Tr \Sigma^{-1}\EE( \xi_{i+1}\otimes \xi_{i+1})}.
\]
We compute
$
\sum_{i=0}^{n-1}\Tr \Sigma^{-1}\EE( \xi_{i+1}\otimes \xi_{i+1}) = \sum_{i=0}^{n-1}\Tr \Sigma^{-1} C = n \Tr \Sigma^{-1} C 
$
and, using Lemma~\ref{lemma:etasemisto}, we bound  $\EE \Vert \mu_{n}-\nu_n\Vert^2_{\Sigma^{-1}}$ as
\[
\frac{{\EE \Vert \mu_{n}-\nu_n\Vert^2_{\Sigma^{-1}}}}{(\gamma n)^2}
\leq
\frac{\Vert \mu_{0}-\nu_0\Vert^2_{\Sigma^{-1}}}{(\gamma n)^2} 
+ {\frac{1}{n}\Tr \Sigma^{-1} C }.
\]
This implies the final bound 
\[
{\EE\Vert \bar \alpha_n- \bar \beta_n  \Vert_\Sigma^2}
\leq
4 \frac{ \Vert \mu_{0}-\nu_0\Vert^2_{\Sigma^{-1}}}{(\gamma n)^2}   
+ {\frac{4}{n}\Tr \Sigma^{-1} C }.
\]
\end{proof}

\subsection{Application of Lemma~\ref{lemma:alphabeta} to prove Proposition~\ref{prop:theta}}

First of all we define the sequence 
\BEQ\label{eq:etastar}
\eta_n^*=\nabla h (\theta_*)-n\gamma \nabla f(\theta_*).
\EEQ
By definition of $\theta_*$, $-\nabla f(\theta_*)\in\partial g(\theta_*)$ then $\eta_n^*\in\partial (h+n\gamma g) \theta_*$ and $ \theta_*=\nabla h_n^*(\eta_n^*)$. Therefore the sequence $\eta_n^*$ is obtained by iterating DA started from the solution of the problem $\theta_*$. 

We note then than Lemma~\ref{lemma:alphabeta} applied to $(\mu_n=\eta_n,\alpha_n=\theta_n)$ and $(\nu_n=\eta_n^*,\beta_n=\theta_*)$ gives the first bound of Proposition~\ref{prop:theta}.

On the other hand, when  considering the noiseless iterates $(\omega_n,\phi_{n})$ defined by  $\omega_n=\omega_{n-1}-\gamma \Sigma (\phi_{n-1}-\tnc) $ and $\phi_n=\nabla h_n^*(\omega_n)$, started from the same point $\phi_0=\theta_0$, we obtain, following Proposition~\ref{prop:phifunction}, for gamma such that $h-\gamma f$ is convex, the bound 
\[
\frac{1}{2}\Vert  \bar \phi_n-\theta_*\Vert_{\Sigma}^{2}
\leq
\psi(\bar \phi_n)- \psi( \theta_*)
\leq   \frac{D_h(\theta_{*},\theta_{0})}{\gamma n}.
\]
Therefore, considering the difference between the semi-stochastic and the noiseless iterate   $(\eta_{n}-\omega_{n})$ which verifies the same equation $\eta_{n}-\omega_{n}=\eta_{n-1}-\omega_{n-1}-\gamma \Sigma (\theta_{n-1}-\phi_{n-1}) +\gamma\xi_n$ with $\theta_0-\phi_0=0$ as initial value, we may apply Lemma~\ref{lemma:alphabeta} to show 
 \[
{\EE\Vert\bar \theta_n- \bar \phi_n  \Vert_\Sigma^2}\leq {\frac{4}{n}\Tr \Sigma^{-1} C }.
\]
And by the  Cauchy-Schwarz inequality 
\BEAS
 {\EE\Vert\bar \theta_n- \theta_*  \Vert_\Sigma^2}
  &\leq&
 2{\EE\Vert\bar \theta_n- \bar \phi_n  \Vert_\Sigma^2}
 +2{\EE\Vert \bar \phi_n  -\theta_* \Vert_\Sigma^2}\\
&\leq&
{\frac{8}{n}\Tr \Sigma^{-1} C}
+4{\frac{D_h(\theta_{*},\theta_{0})}{\gamma n}},
\EEAS
which proves the second bound of Proposition~\ref{prop:theta}.

It is worth noting that the condition on the step-size of Lemma~\ref{lemma:etasemisto} is less restrictive than in Proposition~\ref{prop:theta}. Indeed for all $\gamma$ such that $2h-\gamma f$ is convex, the difference between the dual iterates of the stochastic and deterministic recursions stay close but the deterministic iterates only converge to the solution for $\gamma$ such that $h-\gamma f$ is convex.

\section{Proof of Proposition~\ref{prop:thetasto}}\label{app:proofsto}

In this section, we prove Proposition~\ref{prop:thetasto}. The proof technique is similar to Proposition~\ref{prop:theta} but with the additional difficulty of the multiplicative noise. 

We first note that Assumption \textbf{(A11)} is equivalent by the  Cauchy-Schwarz inequality to 
\BEQ \label{eq:kurtocobain}
\EE \langle x_n,Mx_n\rangle \langle x_n,Nx_n\rangle\leq \kappa \tr (M\Sigma)\tr(N\Sigma),
\EEQ for all positive semi-definite symmetric matrices $M$ and $N$ \citep[see, e.,g., proof in][]{DieFlaBac16}. We will often use, in the following demonstrations, \eq{kurtocobain} and its direct corollary
\BEQ
\langle x_n,Mx_n\rangle  x_n\otimes x_n\preccurlyeq \kappa \tr (M\Sigma)\Sigma,
\EEQ 
without always referring to it. 

\subsection{A simple proof for the bounded constrained case}\label{app:stobounded}

We first prove Proposition~\ref{prop:thetasto} for the constrained case. It is then a simple corollary of Proposition~\ref{prop:theta}.

Let us denote by  $\mathcal{C}$ a  bounded convex set and consider the constrained problem ($g=\mathbf{1}_{\mathcal{C}}$). We remind that the general stochastic oracle for SDA in least-squares regression is 
\[
\nabla f_n(\theta)= (\Sigma+\zeta_n)(\theta-\tnc)-\xi_n, \text{ for } \theta\in{\RR^d},
\]
with $\zeta_n=x_n\otimes x_n-\Sigma$. We denote by $r=\max_{\theta\in\mathcal C}\Vert \theta-\tnc\Vert_2$ and we show that the noise covariance is directly bounded, despite the multiplicative noise: 
\[
\EE \Big[\big(\nabla f_n(\theta)-\nabla f(\theta)\big)\otimes \big(\nabla f_n(\theta)-\nabla f(\theta)\big)\Big]\preccurlyeq 2 \EE \big[\zeta_n (\theta-\tnc)\otimes (\theta-\tnc)\zeta_n\big]+ 2\EE \xi_n\otimes \xi_n,
\]
and using Assumption \textbf{(A11)}
\[
 \EE \big[\zeta_n (\theta-\tnc)\otimes (\theta-\tnc)\zeta_n\big]\preccurlyeq r^2\EE \zeta_n \zeta_n\preccurlyeq r^2 \kappa (\Tr \Sigma) \Sigma.
\]
Therefore 
\[
 \EE \Big[\big(\nabla f_n(\theta)-\nabla f(\theta)\big)\otimes \big(\nabla f_n(\theta)-\nabla f(\theta)\big)\Big]\preccurlyeq 2\big(\sigma^2 + r^2 \kappa (\Tr \Sigma)\big)\Sigma.
\]
Hence Proposition~\ref{prop:theta} already implies  for all step-size such that  $h-\gamma f$ is convex 
\[\frac{1}{2}{\EE \Vert  \bar \theta_n - \theta_*\Vert_{\Sigma}^2}
  \leq
  2 \frac{D_h(\theta_{*},\theta_{0})}{\gamma n} +\frac{8d}{n} (\sigma^2 +\kappa r^2\tr\Sigma).
  \]
  
\subsection{A general result}
We prove in this section a more general result than Proposition~\ref{prop:thetasto} under the additional assumption
\begin{description}
 \item [\textbf{(A14)}]
There exists $b\in[0,1]$ and $\mu_b>0$ such that $h-\frac{\mu_b}{2}\Vert\cdot\Vert_{\Sigma^b}^2$ is convex.
\end{description}
\begin{proposition}\label{prop:thetastogen}
Assume \textbf{(A2-4)} and \textbf{(A7-14)}. Consider the recursion in \eq{daquad}. For any constant step-size $\gamma$  such that $\gamma \leq \min\{ \frac{\mu_b}{4\kappa \tr \Sigma^{1-b}},  \frac{1}{\kappa L d}\}$. Then 
 \begin{multline*}
  \frac{1}{2}{\EE \Vert \bar \theta_{n}-\theta_*\Vert_\Sigma^2}
  \leq
    { 2\frac{D_h(\theta_*,\theta_{0})}{\gamma n}}
      +\frac{24}{n}\tr \Sigma^{-1}C
   +\frac{16\kappa d \gamma }{n \mu_b}\tr C \Sigma^{-b}  
 \\ +\frac{8\kappa d  }{n }\bigg(\frac{4\kappa \gamma \tr \Sigma^{1-b}}{\mu_b}+3\bigg)\Vert\theta_*-\tnc\Vert_\Sigma^2  
   + {  80 \frac{\kappa d}{\gamma n^2}} {D_h(\theta_*,\theta_{0}) }
  +\frac{16\kappa d}{n^2}g(\theta_0).
   \end{multline*}
\end{proposition}
We note Assumption \textbf{(14)} is always satisfied for $b=1$, for which  it is  Assumption \textbf{(13)}. Therefore Proposition  \ref{prop:thetastogen} directly implies Proposition~\ref{prop:thetasto} as a corollary. 
We prove now two auxiliary lemmas which will be used in the proof of the Proposition~\ref{prop:thetastogen}.

\subsection{Two auxiliary results for least-squares objectives}
For $b\in[0,1]$, we denote by $T_b$ the operator $T_b=\EE [\langle x,\Sigma^{-b}x\rangle x\otimes x]$. We first prove that, for least-square objectives, the sum of the function evaluated along the primal iterates remains bounded. 
\begin{lemma}\label{lemma:alphabetafunction}
Let us consider the recursion  $\eta_n=\eta_{n-1}-\gamma  x_n\otimes x_n (\theta_{n-1}-\theta_*) +\gamma\xi_{n}$  and assume $g$ is positive and  there exist $\mu_b$ such that  $h- \frac{\mu_b }{2}\Vert\cdot \Vert_{\Sigma^b}^2$ is convex and $\kappa$ such that $T_b\preccurlyeq\kappa \tr(\Sigma^{1-b})\Sigma$, then for $\gamma\leq \mu_b/(4\kappa \tr \Sigma^{1-b})$ and $\theta\in \mathcal{X}$ we have
\begin{multline*}
\EE \sum_{i=0}^n[\psi(\theta_{i})-\psi(\theta) ]+ \Big(1-4\gamma \kappa \tr (\Sigma^{1-b})/\mu_b\Big)\sum_{i=0}^n\frac{1}{2}\EE\Vert \theta_{i}-\theta\Vert_\Sigma^2
 \\
 \leq 
 \frac{D_h(\theta,\theta_{0})-\EE D_h(\theta,\theta_{n+1}) }{\gamma}+(n+1)\gamma/\mu_b \tr \Sigma^{-b} C+4(n+1)\kappa \tr (\Sigma^{1-b})/\mu_b f(\theta)+g(\theta_0). 
\end{multline*}
\end{lemma}
We note that we can also obtain a bound depending on $2\psi(\theta)$ rather than $4f(\theta)$ with a similar proof.

\begin{proof}
Let denote by $f_n(\theta)= x_n\otimes x_n (\theta-\tnc) +\xi_{n}$. Then following the proof of Proposition~\ref{prop:phifunction} (see \eq{Breglya}) we have the expansion
\begin{multline}\label{eq:dd}
  \tilde D_n(\theta,\eta_n)-\tilde D_{n-1}(\theta,\eta_{n-1})
  \leq
  -\gamma (g(\theta_n)-g(\theta)) -\gamma \langle \nabla f_n(\theta_{n-1}), \theta_{n-1}-\theta \rangle\\
  -D_h(\theta_n,\theta_{n-1})+\gamma \langle \nabla f_n(\theta_{n-1}), \theta_{n-1}-\theta_n \rangle.
\end{multline}
Since $h-\frac{\mu_b}{2}\Vert\cdot \Vert_{\Sigma^b}^2$ is convex, using Proposition~\ref{prop:lc}, we get that $D_h(\theta_n,\theta_{n-1})\geq\frac{\mu_b}{2}\Vert \theta_n-\theta_{n-1}\Vert_{\Sigma^b}^2$. Let denote by $A=-D_h(\theta_n,\theta_{n-1}), 
 +\gamma \langle \nabla f_n(\theta_{n-1}), \theta_{n-1}-\theta_n \rangle$,
 \BEAS 
 A
 &\leq&
 -\frac{\mu_b}{2}\Vert\theta_n-\theta_{n-1}\Vert_{\Sigma^b}^2
 +\gamma \langle x_n\otimes x_n (\theta_{n-1}-\tnc) +\xi_{n}, \theta_{n-1}-\theta_n \rangle\\
  &\leq&
 -\frac{\mu_b }{2}\Vert\theta_n-\theta_{n-1}\Vert_{\Sigma^b}^2\\
 &&
 + \Big \langle \frac{\gamma\Sigma^{-b/2}}{\sqrt{\mu_b}}[ x_n\otimes x_n (\theta_{n-1}-\tnc) +\xi_{n}], {\Sigma^{b/2}} {\sqrt{\mu_b}}\theta_{n-1}-\theta_n \Big\rangle\\
 &\leq&
 \frac{\gamma^2\mu_b}{2  }\Vert x_n\otimes x_n (\theta_{n-1}-\tnc) +\xi_{n}\Vert_{ \Sigma^{-b} }^2 \\
 &&\textstyle
 -\frac{1 }{2}\Vert{\gamma \Sigma^{b/2}}{\sqrt{\mu_b}}(\theta_n-\theta_{n-1})- \frac{\gamma \Sigma^{-b/2}}{\sqrt{\mu_b}}[ x_n\otimes x_n (\theta_{n-1}-\tnc) +\xi_{n}]\Vert_2^2\\
  &\leq&
 \frac{\gamma^2}{2\mu_b}\Vert   x_n\otimes x_n (\theta_{n-1}-\tnc) +\xi_{n}\Vert_{ \Sigma^{-1} }^2\\
 &\leq&
 \frac{\gamma^2}{\mu_b} \Vert  \theta_{n-1}-\tnc\Vert_{T_b}^2
 +  \frac{\gamma^2}{\mu_b}  \Vert  \xi_{n}\Vert_{ \Sigma^{-b} }^2.
 \EEAS
 Thus, taking the conditional expectation and assuming that $\kappa$ is such that $T_b\preccurlyeq\kappa \tr (\Sigma^{1-b}) \Sigma$ we obtain
 \begin{multline*}
   -\EE [D_h(\theta_n,\theta_{n-1})\mathcal{F}_{n-1}]
 +\gamma \EE \langle \nabla f_n(\theta_{n-1}), \theta_{n-1}-\theta_n \rangle\mathcal{F}_{n-1}]\\
 \leq 
 \frac{\gamma^2\kappa \tr (\Sigma^{1-b}) }{\mu_b} \Vert  \theta_{n-1}-\tnc\Vert_{ \Sigma }^2
+ \frac{\gamma^2}{\mu_b} \tr \Sigma^{-b} C.
 \end{multline*}
Taking again the conditional expectation in \eq{dd}, we have for $\theta\in\mathcal{X}$
 \BEAS
 \EE [\tilde D_n(\theta,\eta_n)\vert \mathcal{F}_{n-1}] -\tilde D_{n-1}(\theta,\eta_{n-1})
 &\leq&
 \frac{\gamma^2 \kappa}{\mu_b} \tr (\Sigma^{1-b})  \Vert  \theta_{n-1}-\tnc\Vert_{ \Sigma }^2
+ \frac{\gamma^2}{{\mu_b}} \tr \Sigma^{-b} C\\
&&
 -\gamma \EE [\langle x_n\otimes x_n (\theta_{n-1}-\tnc) +\xi_{n}, \theta_{n-1}-\theta \rangle \vert \mathcal{F}_{n-1}]\\
 &&
 -\gamma(\EE [g(\theta_n)\vert \mathcal{F}_{n-1}]-g(\theta))\\
 &\leq&
   \frac{\gamma^2 \kappa} {\mu_b} \tr (\Sigma^{1-b})  \Vert  \theta_{n-1}-\tnc\Vert_{ \Sigma }^2
      -\gamma \langle \theta_{n-1}-\tnc,\Sigma(\theta_{n-1}-\theta)\rangle
\\
   &&
      + \frac{\gamma^2}{{\mu_b}} \tr \Sigma^{-b} C
    -\gamma(\EE [g(\theta_n)\vert \mathcal{F}_{n-1}]-g(\theta)) .
 \EEAS  
 And we note that
\[                                                                                                                                                                                                                                                                -\gamma \langle \theta_{n-1}-\tnc,\Sigma(\theta_{n-1}-\theta_*)\rangle=  -\gamma[f(\theta_{n-1})-f(\theta_*)]-\frac{\gamma}{2}\Vert \theta_{n-1}-\theta\Vert_\Sigma^2.                                                                                                                                                                                                                                                      \]
Therefore 
\BEAS
 \EE [\tilde D_n(\theta_*,\eta_n)\vert \mathcal{F}_{n-1}] -\tilde D_{n-1}(\theta_*,\eta_{n-1})
 &\leq&
 -\gamma[f(\theta_{n-1})-f(\theta_*)+\EE [g(\theta_n)\vert \mathcal{F}_{n-1}]-g(\theta_*) ]\\
 && 
-\frac{\gamma}{2} \Big(1-4 \frac{\gamma \kappa}{\mu_b} \tr (\Sigma^{1-b})\Big)\Vert \theta_{n-1}-\tnc\Vert_\Sigma^2\\
&&
+2\frac{\gamma^2 \kappa}{\mu_b} \tr (\Sigma^{1-b}) \Vert \theta-\tnc\Vert_\Sigma^2
+  \frac{\gamma^2}{\mu_b} \tr \Sigma^{-b} C.
\EEAS
Taking the total expectation  we obtain
\begin{multline*}
 \EE f(\theta_{n-1})-f(\theta_*)+\EE g(\theta_n)-g(\theta_*) ]
+\frac{1}{2} \Big(1-4\frac{\gamma \kappa}{\mu_b} \tr (\Sigma^{1-b})\Big)\Vert \theta_{n-1}-\tnc\Vert_\Sigma^2\\
 \leq \frac{ \EE \tilde D_{n-1}(\theta_*,\eta_{n-1})-\EE \tilde D_n(\theta_*,\eta_{n}) }{\gamma}+2\frac{\gamma \kappa}{\mu_b}\tr (\Sigma^{1-b}) \Vert \theta-\tnc\Vert_\Sigma^2 + \frac{\gamma}{\mu_b} \tr \Sigma^{-b} C,
\end{multline*}
which, summing from $i=0$ to $i=n$, leads to
\begin{multline*}
 \sum_{i=0}^n[\EE f(\theta_{i})-f(\theta_*)+\EE g(\theta_i)-g(\theta_*) ]]+ \Big(1-4\frac{\gamma \kappa}{\mu_b} \tr (\Sigma^{1-b}) \Big)\sum_{i=0}^n\frac{1}{2}\Vert \theta_{i}-\tnc\Vert_\Sigma^2
 \leq 
 \\
 \frac{D_h(\theta_*,\theta_{0})-\EE \tilde D_{n+1}(\theta_*,\eta_{n+1}) }{\gamma}+ 4\frac{\gamma \kappa}{\mu_b} \tr (\Sigma^{1-b}) (n+1)\Vert \theta-\tnc\Vert_\Sigma^2
 \\+(n+1)\frac{\gamma }{\mu_b} \tr \Sigma^{-b} C -\EE g(\theta_{n+1})+g(\theta_0). 
\end{multline*}
The result follows if $g$ is non negative.
\end{proof}
We now present an extension of Lemma~\ref{lemma:etasemisto} to least-squares objectives.
\begin{lemma}\label{lemma:etasto}
Let us consider two sequences of iterates $(\mu_k,\alpha_k)$ and $(\nu_k,\beta_k)$ which satisfy the recursion  $\mu_n-\nu_n=\mu_{n-1}-\nu_{n-1}-\gamma x_n\otimes x_n (\alpha_{n-1}-\beta_{n-1}) +\gamma\xi_n$, $\alpha_n=\nabla h_n^*(\mu_n)$ and $\beta_n=\nabla h_n^*(\nu_n)$  and denote by $C=\EE[x_n\otimes x_n]$ for $n\geq 0$. Assume that $\gamma$ is such that $h-\gamma T$ is convex. Then
\[
\EE \Vert \mu_{n}-\nu_{n}\Vert_{\Sigma^{-1}}^2
\leq 
\EE \Vert \mu_{0}-\nu_{0}\Vert_{\Sigma^{-1}}^2 
+2\gamma^2 n \Tr \Sigma^{-1} C.
\]
\end{lemma}

We note that the condition $h-\gamma T$ is rather restrictive since bounds on $T$ are often of the form $d$ times a matrix. For instance \eq{kurtocobain} directly implies $T\preccurlyeq \kappa d \Sigma$. Even for independent normal data $x_n$ with diagonal covariance matrix $\Sigma$ we are able to derive  the equality $T=(d+2)\Sigma$.
\begin{proof}
We expand
\BEAS
\Vert \mu_{n}-\nu_{n}\Vert_{\Sigma^{-1}}^2
&=&
\Vert \mu_{n-1}-\nu_{n-1}\Vert_{\Sigma^{-1}}^2
+ \gamma^2 \Vert x_n\otimes x_n (\alpha_{n-1}-\beta_{n-1})+\xi_{n}\Vert_{\Sigma^{-1}}^2 \\
&&
-2\gamma \langle  x_n\otimes x_n (\alpha_{n-1}-\beta_{n-1}) +\xi_{n},\Sigma^{-1} (\mu_{n-1}-\nu_{n-1})\rangle.
\EEAS
Taking conditional expectations, we get
\BEAS
 \EE[\Vert \mu_{n}-\nu_{n}\Vert_{\Sigma^{-1}}^2 \vert \mathcal{F}_{n-1}]
 &=& 
\Vert \mu_{n-1}-\nu_{n-1}\Vert_{\Sigma^{-1}}^2
 +\gamma^2\EE[ \Vert x_n\otimes x_n (\alpha_{n-1}-\beta_{n-1})+\xi_{n}\Vert_{\Sigma^{-1}}^2 \vert \mathcal{F}_{n-1}]\\
 &&
 -2\gamma \EE[\langle  x_n\otimes x_n (\alpha_{n-1}-\beta_{n-1}) +\xi_{n},\Sigma^{-1} (\mu_{n-1}-\nu_{n-1})\rangle\vert \mathcal{F}_{n-1}]\\
 &=&
\Vert \mu_{n-1}-\nu_{n-1}\Vert_{\Sigma^{-1}}^2
 +\gamma^2\EE[ \Vert x_n\otimes x_n (\alpha_{n-1}-\beta_{n-1})+\xi_{n}\Vert_{\Sigma^{-1}}^2 \vert \mathcal{F}_{n-1}]\\
 &&
 -2\gamma \langle  \alpha_{n-1}-\beta_{n-1} ,\mu_{n-1}-\nu_{n-1}\rangle.
\EEAS
Using $(a+b)^2\leq 2a^2+2b^2$ and denoting by $B= \EE[ \Vert x_n\otimes x_n (\alpha_{n-1}-\beta_{n-1})+\xi_{n}\Vert_{\Sigma^{-1}}^2 \vert \mathcal{F}_{n-1}]$, this leads to
\BEAS
B
&\leq&
2\EE[ \Vert x_n\otimes x_n (\alpha_{n-1}-\beta_{n-1})\Vert_{\Sigma^{-1}}^2 \vert \mathcal{F}_{n-1}]
+2\EE[ \Vert \xi_{n}\Vert_{\Sigma^{-1}}^2 \vert \mathcal{F}_{n-1}]\\
&\leq&
 2 \Vert \alpha_{n-1}-\beta_{n-1}\Vert_{\EE[x_n\otimes x_n \Sigma^{-1} x_n\otimes x_n \vert \mathcal{F}_{n-1}]
}^2 
+2\tr \Sigma^{-1}\EE[\epsilon_{n}\otimes \epsilon_{n} \vert \mathcal{F}_{n-1}]\\
&\leq&
2\Vert \alpha_{n-1}-\beta_{n-1}\Vert_T^2 +2\tr \Sigma^{-1}C,
\EEAS
with $T= \EE[x\otimes x\Sigma^{-1}  x\otimes x]$. Thus we obtain
\BEAS
 \EE[\Vert \mu_{n}-\nu_{n}\Vert_{\Sigma^{-1}}^2 \vert \mathcal{F}_{n-1}]
 &\leq&
\Vert \mu_{n-1}-\nu_{n-1}\Vert_{\Sigma^{-1}}^2
+2\gamma^2\tr \Sigma^{-1}C \\
 &&
 -2\gamma \langle\mu_{n-1}-\nu_{n-1}-\gamma T(\alpha_{n-1}-\beta_{n-1}), \alpha_{n-1}-\beta_{n-1} \rangle\\
 &\leq&
 \Vert \mu_{n-1}-\nu_{n-1}\Vert_{\Sigma^{-1}}^2
+2\gamma^2\tr \Sigma^{-1}C,
\EEAS
assuming that $\gamma$ is such $h-\gamma \frac{1}{2}\Vert \cdot \Vert_T^2$ is convex (as in the proof of Lemma~\ref{lemma:etasemisto}). Taking global expectations,  we have shown that 
\[
\EE \Vert \mu_{n}-\nu_{n}\Vert_{\Sigma^{-1}}^2
\leq   \EE \Vert \mu_{n-1}-\nu_{n-1}\Vert_{\Sigma^{-1}}^2+2\gamma^2\tr \Sigma^{-1} C.
\]
\end{proof}
\subsection{Bound on the difference between two averages of primal variables}
We present now the following lemma with is an analogue of Lemma~\ref{lemma:alphabeta} for the least-squares problem. It shows that the difference between the average of two sequences of primal iterates which follow the same recursion is  $O(1/n)$.
\begin{lemma}\label{lemma:alphabetasto}
Let us consider two sequences of iterates $(\mu_k,\alpha_k)$ and $(\nu_k,\beta_k)$ which satisfy the recursion  $\mu_n-\nu_n=\mu_{n-1}-\nu_{n-1}-\gamma x_n\otimes x_n (\alpha_{n-1}-\beta_{n-1}) +\gamma\xi_n$, $\alpha_n=\nabla h_n^*(\mu_n)$ and $\beta_n=\nabla h_n^*(\nu_n)$. For $n\geq0$ denote by $C=\EE[x_n\otimes x_n]$.  Assume that $\gamma$ is such that $h-\gamma T$ is convex and there exists $\kappa$ such that $ T\preccurlyeq \kappa d \Sigma$. Then
\[
{\EE \Vert\bar \alpha_{n}-\bar \beta_{n}\Vert_\Sigma^2}
\leq
4{ \frac{\kappa d  -1}{ n^2}\sum_{i=0}^{n-1} \EE \Vert \alpha_{i}-\beta_{i}\Vert_\Sigma^2}
+4 \frac{\Vert \eta_0-\mu_0\Vert^2_{\Sigma^{-1}}}{(\gamma n)^2}
+{\frac{8}{n}\Tr \Sigma^{-1} C }.
\]

\end{lemma}

\begin{proof}
Using the expansion $\mu_n-\nu_n=\mu_{n-1}-\nu_{n-1}-\gamma x_n\otimes x_n (\alpha_{n-1}-\beta_{n-1}) +\gamma\xi_n$, we derive
\BEAS
\Sigma(\alpha_n-\beta_n)
&=&
(\Sigma-x_{n+1}\otimes x_{n+1})(\alpha_n-\beta_n)
+x_{n+1}\otimes x_{n+1}(\alpha_n-\beta_n)\\
&=&
(\Sigma-x_{n+1}\otimes x_{n+1})(\alpha_n-\beta_n)+\frac{\mu_n-\nu_n-\mu_{n+1}+\nu_{n+1}}{\gamma}+\xi_{n+1}.
\EEAS
We obtain by summing $n$ times
\BEAS
\Sigma^{1/2} \sum_{i=0}^{n-1}(\alpha_i-\beta_i)&=&\sum_{i=0}^{n-1}\Sigma^{-1/2} X_{i+1}+\frac{\Sigma^{-1/2} (\mu_0-\nu_0-\mu_n+\nu_n)}{\gamma}+\sum_{i=0}^{n-1}\Sigma^{-1/2} \xi_{i+1},
\EEAS
where we denote by $X_i=(\Sigma-x_{i}\otimes x_{i})(\alpha_{i-1}-\beta_{i-1})$ which is a square-integrable martingale difference sequence.
We use $(a+b)^2\leq2(a^2+b^2)$ to obtain
\[
 \Vert  (\bar \alpha_{n}-\bar \beta_{n})\Vert_\Sigma^2 
 \leq
 \frac{2}{n^2} \Big\Vert\sum_{i=0}^{n-1}\Sigma^{-1/2} X_{i+1}+\frac{\Sigma^{-1/2} (\mu_0-\nu_0)}{\gamma}+\sum_{i=0}^{n-1}\Sigma^{-1/2} \xi_{i+1} \Big\Vert_2^2
 +2 \frac{\Vert \mu_n-\nu_n \Vert^2_{\Sigma^{-1}}}{(\gamma n)^2}.
\]
Therefore using martingale square moment inequalities which here amount to considering the variance of the sum as the sum of the variance, we have
\begin{multline}\label{eq:mink}
  {\EE \Vert  (\bar \alpha_{n}-\bar \beta_{n})\Vert_\Sigma^2}
  \leq
  \frac{4}{n^2}{\sum_{i=0}^{n-1} \EE \Vert  X_{i+1}\Vert_{\Sigma^{-1}}^2}
  + 2\frac{\Vert\mu_0-\nu_0\Vert^2_{\Sigma^{-1}}}{(\gamma n)^2} \\
  +2\frac{{\EE \Vert \mu_n-\nu_n \Vert^2_{\Sigma^{-1}}}}{(\gamma n)^2}  +\frac{4}{n^2}{\sum_{i=0}^{n-1}\Tr \Sigma^{-1}\EE( \xi_{i+1}\otimes \xi_{i+1})}.
\end{multline}
\begin{itemize}
 \item
The variance term may be bounded as
\[
\sum_{i=0}^{n-1}\Tr \Sigma^{-1}\EE( \xi_{i+1}\otimes \xi_{i+1}) \leq \sum_{i=0}^{n-1}\Tr \Sigma^{-1} C \leq n \Tr \Sigma^{-1} C. 
\]
\item
Following Lemma~\ref{lemma:etasto} we bound the dual iterates $\EE \Vert \mu_n-\nu_n \Vert^2_{\Sigma^{-1}}$ as
\[
\frac{{\EE \Vert \mu_n-\nu_n \Vert^2_{\Sigma^{-1}}}}{(\gamma n)^2}
\leq
 \frac{\Vert\mu_0-\nu_0\Vert^2_{\Sigma^{-1}}}{(\gamma n)^2} 
+ {\frac{2}{n}\Tr \Sigma^{-1} C }.
\]
\item
The martingale difference sequence $(X_i)$ satisfies
\BEAS
 \EE \Vert \Sigma^{-1/2}  X_{i+1}\Vert_2^2 
 &\leq &
 \EE \langle (\Sigma-x_{i+1}\otimes x_{i+1})(\alpha_{i}-\beta_{i}),\Sigma^{-1}   (\Sigma-x_{i+1}\otimes x_{i+1})(\alpha_{i}-\beta_{i})\rangle \\
 &\leq &
 \langle \alpha_{i}-\beta_{i}, \EE [(\Sigma-x_{i+1}\otimes x_{i+1})^\top \Sigma^{-1}(\Sigma-x_{i+1}\otimes x_{i+1})] (\alpha_{i}-\beta_{i})\rangle\\
 &\leq &
 \langle \alpha_{i}-\beta_{i}, [\EE (x_{i+1}\otimes x_{i+1})^\top\Sigma^{-1}x_{i+1}\otimes x_{i+1}-\Sigma ] (\alpha_{i}-\beta_{i})\rangle\\ 
  &\leq &
  \langle \alpha_{i}-\beta_{i},[ T-\Sigma ](\alpha_{i}-\beta_{i})\rangle \\
  &\leq&(\kappa d -1) \Vert \alpha_{i}-\beta_{i}\Vert^2_{\Sigma}.
\EEAS
\end{itemize}
Consequently we obtain in \eq{mink}
\[
\EE \Vert \Sigma^{1/2} (\bar \alpha_{n}-\bar \beta_{n})\Vert_2^2
\leq
4{ \frac{\kappa d-1}{ n^2}\sum_{i=0}^{n-1} \EE \lVert \alpha_{i}-\beta_{i}\Vert^2_\Sigma}
+4 \frac{\Vert \mu_0-\nu_0\Vert^2_{\Sigma^{-1}}}{(\gamma n)^2}
+\frac{8}{n}\Tr \Sigma^{-1} C .
\]
\end{proof}

\subsection{Application of Lemma~\ref{lemma:alphabetasto} to the proof of Proposition~\ref{prop:thetastogen}}

We are now able to prove Proposition~\ref{prop:thetastogen} using Lemma~\ref{lemma:alphabetasto}. 

Firstly we can directly apply  Lemma~\ref{lemma:alphabetasto} to $(\mu_n=\eta_n,\alpha_n=\theta_n)$ and $(\nu_n=\eta_n^*,\beta_n=\theta_*)$ where $(\eta_n^*,\theta_*)$ are defined in \eq{etastar}. This implies 
\[
{\EE \Vert \Sigma^{1/2} (\bar \theta_{n}- \theta_*)\Vert_2^2}
\leq
4{ \frac{\kappa d-1}{ n^2}\sum_{i=0}^{n-1} \EE \lVert \theta_{i}-\theta_{*}\Vert^2_\Sigma}
+4 \frac{\Vert \nabla h(\theta_0)-\nabla h(\theta_*)\Vert^2_{\Sigma^{-1}}}{(\gamma n)^2}
+{\frac{8}{n}\Tr \Sigma^{-1} C }.
\]
Following Lemma~\ref{lemma:alphabetafunction}, the primal variables $(\theta_i)$ satisfy
\[
 \sum_{i=0}^{n-1} \EE \lVert \theta_{i}-\theta_{*}\Vert^2_\Sigma\leq 2\frac{D_h(\theta_*,\theta_{0}) }{\gamma}+2\frac{n\gamma}{\mu_b} \tr \Sigma^{-b}C+\frac{8n\gamma \kappa \tr \Sigma^{1-b}}{\mu_b}f(\theta_*)+2{g(\theta_0)}.
\]
This leads to the final bound
\begin{multline}
  {\EE \Vert \Sigma^{1/2} (\bar\theta_{n}-\theta_{*})\Vert_2^2}
   \leq
  {8\frac{\kappa d-1}{ \gamma n^2}D_h(\theta_*,\theta_0)}
      +4 \frac{ \Vert \nabla h(\theta_0)-\nabla h(\theta_*)\Vert^2_{\Sigma^{-1}}}{(\gamma n)^2}
      \\
      +8  {\frac{1}{n}\Tr \Sigma^{-1} C }
      +8\frac{\kappa d -1 }{n}\frac{\gamma}{\mu_b} [\tr \Sigma^{-b}C+4\kappa \tr \Sigma ^{1-b} f(\theta_*)]     + {8\frac{\kappa d-1}{  n^2}g(\theta_0)}.
\end{multline}
This bound depends on $\Vert \cdot \Vert_{\Sigma^{-1}}$ which may be infinite. For this reason we compare again  the noisy iterate $\theta_n$ to the noiseless iterate we still denote by $(\phi_n)$. We remind these iterates verify the recursion 
\[
\omega_n=\omega_{n-1}-\gamma \Sigma (\phi_{n-1}-\tnc).
\]
Therefore the difference $(\eta_n-\omega_n)$ satisfies the same form of recursion as $(\eta_n)$:
\[
\eta_n-\omega_n=\nabla \eta_{n-1}-\omega_{n-1}-\gamma x_n\otimes x_n (\theta_{n-1}-\phi_{n-1}) +\gamma\epsilon_n,
\]
with a different noise $\epsilon_n=\xi_n-  [x_n\otimes x_n -\Sigma] (\phi_{n-1}-\tnc)$ and  $0$ for initial value. Although the noise $\epsilon_n$ is different from $\xi_n$, its covariance  is still bounded by
\BEAS
 \frac{1}{3}\EE[\epsilon_n\otimes\epsilon_n]
 &\preccurlyeq &
   \EE[\xi_n\otimes\xi_n]
 + \EE[ [x_n\otimes x_n -\Sigma](\phi_{n-1}-\theta_*) \otimes (\phi_{n-1}-\theta_*)  [x_n\otimes x_n -\Sigma] ] \\
 && + \EE[ [x_n\otimes x_n -\Sigma](\theta_*-\tnc) \otimes (\theta_*-\tnc)  [x_n\otimes x_n -\Sigma] ] \\
  &\preccurlyeq &
  \EE[\xi_n\otimes\xi_n]
- \EE[ \Sigma(\phi_{n-1}-\theta_*)^{\otimes 2}\Sigma]
- \EE[ \Sigma(\theta_*-\tnc)^{\otimes 2}\Sigma]
 \\
 &&
 + \EE[ x_n\otimes x_n(\phi_{n-1}-\theta_*)^{\otimes 2}x_n\otimes x_n]
 + \EE[ x_n\otimes x_n(\theta_*-\tnc)^{\otimes 2}x_n\otimes x_n]
\\
&\preccurlyeq &
  \EE[\xi_n\otimes\xi_n]
 + (\kappa-1) \big(\Vert \phi_{n-1}-\theta_*\Vert_\Sigma^2 +\Vert \theta_*-\tnc\Vert_\Sigma^2\big) \Sigma,
\EEAS 
where we have use that for $z\in{\RR^d}$, $\EE \langle z,x_n\rangle^4\leq \kappa \langle z,\Sigma z\rangle$.
We may apply Proposition~\ref{prop:phifunction} and  obtain
\[
 \EE[\epsilon_n\otimes\epsilon_n] \preccurlyeq  3 C +\frac{6(\kappa-1)}{\gamma n} D_h(\theta_*,\theta_0)  \Sigma+6(\kappa-1)f(\theta_*) .
\]
Thereby Lemma~\ref{lemma:alphabetasto} can be applied with $\theta_{0}=\alpha_{0}$ and we get
\[
 {\EE \Vert \bar \theta_{n}-\bar \phi_{n}\Vert_\Sigma^2}
  \leq
  4 { \frac{\kappa d-1}{ n^2}\sum_{i=0}^{n-1} \EE \Vert \theta_{i}-\phi_{i}\Vert_\Sigma ^2}  +{\frac{8}{n}\Tr \Sigma^{-1}\EE[\epsilon_n\otimes\epsilon_n] }.
\]
As before we  apply Lemma~\ref{lemma:alphabetafunction} to have
\BEAS
  \sum_{i=0}^{n-1} \EE \Vert \theta_{i}-\phi_{i}\Vert_\Sigma ^2
     &\leq&
    \Bigg[ { 2\sum_{i=0}^{n-1}  \EE \Vert \theta_{i}-\theta_*\Vert_\Sigma ^2+2\sum_{i=0}^{n-1} \Vert \phi_{i}-\theta_*\Vert_\Sigma ^2}\Bigg]\\
    &\leq&
       \Bigg[ \frac{8D_h(\theta_*,\theta_{0}) }{\gamma}+\frac{n\gamma}{\mu_b} 4\tr \Sigma^{-b}C+\frac{16n\gamma \kappa \tr \Sigma^{1-b}}{\mu_b}f(\theta_*)+4{g(\theta_0)}\Bigg].
       \EEAS     
 Therefore
 \begin{multline*}
 {\EE \Vert \bar \theta_{n}-\bar \phi_{n}\Vert_\Sigma^2}
  \leq
  {  4 \frac{\kappa d-1}{n^2}}\Bigg[ 8\frac{D_h(\theta_*,\theta_{0}) }{\gamma}+\frac{n\gamma}{\mu_b} 4\tr \Sigma^{-b}C+\frac{16n\gamma \kappa \tr \Sigma^{1-b}}{\mu_b}f(\theta_*)+4{g(\theta_0)}\Bigg]\\
  +\frac{8}{{n}}\Bigg[{ 3\tr\Sigma^{-1}C +\frac{6(\kappa-1)}{\gamma n} D_h(\theta_*,\theta_0)d+6(\kappa-1)f(\theta_*)d}\Bigg].
\end{multline*}
And rearranging terms we obtain 
 \begin{multline*}
 {\EE \Vert \bar \theta_{n}-\bar \phi_{n}\Vert_\Sigma^2}
  \leq
  {  80 \frac{\kappa d}{\gamma n^2}} {D_h(\theta_*,\theta_{0}) }
  +\frac{64\kappa^2 d \gamma }{n \mu_b}\tr \Sigma^{1-b}f(\theta_*)
  +\frac{16\kappa d \gamma }{n \mu_b}\tr C \Sigma^{-b}   \\
  +\frac{16\kappa d}{n^2}g(\theta_0)
  +\frac{24}{n}\tr \Sigma^{-1}C
  +\frac{48\kappa d }{n}f(\theta_*).
\end{multline*}
And by the  Cauchy-Schwarz inequality (${\EE \Vert \bar \theta_{n}-\theta_*\Vert_\Sigma^2}
   \leq
   2{\EE \Vert \bar \theta_{n}-\bar\phi_n\Vert_\Sigma^2}+ 2{\EE \Vert\bar\phi_n-\theta_*\Vert_\Sigma^2}$)
    \begin{multline*}
 \frac{1}{2}{\EE \Vert \bar \theta_{n}-\theta_*\Vert_\Sigma^2}
  \leq
    { 2\frac{D_h(\theta_*,\theta_{0})}{\gamma n}}
      +\frac{24}{n}\tr \Sigma^{-1}C
   +\frac{16\kappa d \gamma }{n \mu_b}\tr C \Sigma^{-b}  
  +\frac{16\kappa d  }{n }\Big(\frac{4\kappa \gamma \tr \Sigma^{1-b}}{\mu_b}+3\Big)f(\theta_*)
  \\
   + {  80 \frac{\kappa d}{\gamma n^2}} {D_h(\theta_*,\theta_{0}) }
  +\frac{16\kappa d}{n^2}g(\theta_0),
\end{multline*}
which proves the second bound of Proposition~\ref{prop:thetastogen}.

\subsection{A corollary of Proposition~\ref{prop:thetastogen} for $h$ with an Euclidean behavior}

When $h$ rather behaves as an Euclidean norm, we may replace Assumptions \textbf{(A12-13)} by the following:
\begin{description}
 \item [\textbf{(A12')}] There exists $\mu_h>0$ such that $h-\frac{\mu_h}{2}\Vert\cdot\Vert_2^2$ is convex. 
 \item [\textbf{(A13 ')}] There exists $R^2$ such that $\EE [\Vert x_n\Vert_2^2 x_n\otimes x_n]\preccurlyeq R^2 \Sigma$.
\end{description}
And  Proposition  \ref{prop:thetastogen} implies the following corollary.
\begin{corollary}\label{cor:thetaeucli}
 Assume For any constant step-size $\gamma$  such that $\gamma \leq \min\{ \frac{\mu_h}{4\kappa R^2},  \frac{R^2}{4\kappa  d}\}$. Then 
 \begin{multline*}
  \frac{1}{2}{\EE \Vert \bar \theta_{n}-\theta_*\Vert_\Sigma^2}
  \leq
    { 2\frac{D_h(\theta_*,\theta_{0})}{\gamma n}}
  +\frac{8}{n}\Big(3+\frac{4\gamma \kappa R^2 }{\mu_h}\Big)\Big(\sigma^2d+\kappa d \Vert  \theta_{*}-\tnc\Vert_\Sigma^2\Big)
  +\frac{16\kappa d }{n^2}\Big( \frac{5D_h(\theta_*,\theta_{0})}{\gamma }+g(\theta_0)\Big).
   \end{multline*}
\end{corollary}
This corollary would pave the way for a general result for larger step-size $\gamma$  without the condition $\gamma \leq    \frac{R^2}{4\kappa  d}$. Unfortunately the latter seems not improvable, as noted after Lemma  \ref{lemma:etasto}. 
\section{Lower bound for non-strongly convex quadratic regularization}\label{app:lowerquad}

We derive, in this section, a lower bound on the performance of SDA when $f$ is the linear form $f(\theta)=\langle a,\theta\rangle$ with $a\in{\RR^d}$ and $g$ is a non-strongly convex quadratic function. 
We assume that the vector $a$ is not available and we only have access to estimates of the gradient 
\begin{equation}\label{eq:oralin}
\nabla f_n(\theta)=a+\xi_n \text{ for } n\geq1,
\end{equation}
 where  $(\xi_n)$ is an uncorrelated zero-mean noise sequence  with bounded covariance.
\begin{proposition}\label{prop:lowerquad}
 For any $d\geq2$, $L>0$; $\gamma>0$ and  finite time horizon $N\geq1$, there exists a quadratic function $g$ $L$-smooth such that for any  uncorrelated zero-mean noise sequence $(\xi_n)$  with bounded covariance $\EE[\xi_n\otimes\xi_n]=\sigma^{2}L\idm_{d}$,  SDA with constant step-size $\gamma$ applied  with the oracle \eq{oralin}
 satisfies
 \[
  \psi(\bar \theta_N)-\psi(\theta_*)\geq \frac{\sigma^2 }{12}\min\{(L\gamma)^2,1\}.
 \]
\end{proposition}
\begin{proof}
For sake of clarity, we consider $d=2$ and $a=0$. Thus $f(\theta)=\EE \langle \xi_n,\theta\rangle=0$. Let  $g(\theta)=\frac{1}{2}\langle \theta,A\theta\rangle$
be a quadratic form with $A=\begin{pmatrix} L & 0\\ 0& \mu \end{pmatrix}$ for $L\geq\mu>0$ with $\mu$ possibly arbitrary small. The noise $(\xi_n)$ is assumed to be uncorrelated zero-mean with bounded covariance $\EE[\xi_n\otimes\xi_n]=\sigma^{2}L\idm_{2}$. The stochastic dual algorithm with step-size $\gamma$ takes the form:
\BEAS
 \theta_n
 &=&
\nabla h_n^*(-n\gamma \bar \xi_n)\\
&=& \argmin_{\theta\in{\RR^d}} \Big\{ \langle \bar \xi_n ,\theta\rangle+\frac{1}{2}\langle \theta,A\theta\rangle+\frac{1}{2n\gamma }\Vert \theta \Vert_2^2\Big\}\\
&=&\gamma n(\idm +\gamma nA)^{-1}\bar \xi_n.
\EEAS
And 
\BEAS
\bar \theta_n
&=&
\frac{\gamma}{n}\sum_{k=1}^{n-1} \sum_{j=1}^k k(\idm +\gamma kA)^{-1}\frac{1}{k}\xi_j \\
&=&
\frac{\gamma}{n}\sum_{j=1}^{n-1} \Big(\sum_{k=j}^{n-1} (\idm+\gamma kA)^{-1}\Big)\xi_j.
\EEAS
Therefore using standard martingale square moment inequalities 
\BEAS
 \EE \langle\bar \theta_n, A \bar \theta_n\rangle  
 &=&
 \frac{\gamma^{2}}{n^2} \sum_{j=1}^n \EE \Big\langle \xi_j \Big(\sum_{k=j}^n (\idm+\gamma kA)^{-1}\Big),A \Big(\sum_{k=j}^n (\idm+\gamma kA)^{-1}\Big)\xi_j\Big\rangle\\
 &=& 
 \frac{\gamma^{2}\sigma^{2}L}{n^2}  \tr  \sum_{j=1}^n  \Big(\sum_{k=j}^n (\idm+\gamma kA)^{-1}\Big)A \Big(\sum_{k=j}^n (\idm+\gamma kA)^{-1}\Big) \idm_{2}\\
 &=& 
   \frac{\gamma^{2}\sigma^{2}L}{n^2} \sum_{j=1}^n  \Big[ L \Big(\sum_{k=j}^n \frac{1}{1 +\gamma L k}\Big)^{2}+ \mu \Big(\sum_{k=j}^n \frac{1}{1 +\gamma \mu k}\Big)^{2}\Big].
 \EEAS 
 And 
 \BEAS
 \EE \langle\bar \theta_n, A \bar \theta_n\rangle  
   &\geq&
   \frac{\gamma^{2}\sigma^{2}L}{n^2}   \Big[  \frac{L}{(1 +\gamma L n)^{2}}+ \frac{\mu}{(1 +\gamma \mu  n)^{2}}\Big] \sum_{j=1}^n (n-j)^{2} \\ 
   &\geq&
   \frac{n\sigma^{2}\gamma^{2}L}{3}   \Big[  \frac{L}{(1 +\gamma L n)^{2}}+ \frac{\mu}{(1 +\gamma \mu  n)^{2}}\Big]    \geq  \frac{n\sigma^{2}\gamma^{2}}{3} \frac{\mu}{(1 +\gamma \mu  n)^{2}} \\
   &\geq&
   \frac{\sigma^{2}L}{12} \min \Big( n \mu \gamma^{2},\frac{1}{\mu n } \Big) .
\EEAS
Conclude by taking $\mu=L/N$.

The proof is the same for $d\geq2$ by considering $A=\diag(L,\dots,L,L\mu)$ with $d-1$ $L$.
\end{proof}

\section{Lower bound for stochastic approximation problems}\label{app:loweroracle}

In this section we relate the problem of aggregation of estimators to the stochastic convex optimization problem, i.e., minimizing a convex function, given only unbiased estimates of its gradients. We will consider the regression and the classification with hinge loss problems which will individually provide lower bounds for quadratic and linear functions. We follow here \citet{tsybakov2003optimal,Lec06,agarwal2010information}.

\subsection{Oracle complexity of stochastic convex optimization}

Beforehand we describe the stochastic oracle model formalism as done by \citet{NemYud83, agarwal2010information, RagRak11}. For a given class of problems we aim to determine lower bounds on the number of queries to a stochastic first-order oracle  needed to optimize to a certain precision any function in this class. To this end we have the following definition.
\begin{definition}[\cite{agarwal2010information}]
For a given constraint convex set $\mathcal C$, and a function class $\mathcal S$, a first-order stochastic oracle is a random mapping $\pi: \mathcal C \times \mathcal S \to \RR\times {\RR^d} $ of the form 
\[
 \phi(\theta, f)=(\tilde f(\theta), g(\theta)), 
\]
such that 
\[
 \EE \tilde f (\theta)=f(\theta); \qquad \ \EE g(\theta)=\nabla  f(\theta),
\]
and there exists a constant $C<\infty$ such that for every $\theta\in{\RR^d}$
\[
 \EE[\Vert  g(\theta)- \nabla f(\theta) \Vert^2]\leq C(1+\Vert\theta\Vert^2).
\]

\end{definition}
The class of first-order stochastic oracle is denoted by $\Phi $. A stochastic approximation algorithm $M$ is a method which approximately minimizes a function $f$ by querying, at each iteration $i$, the oracle at the point $\theta_i$. The oracle answers with the information $\phi(\theta_i,f)$ and the method uses all the information $\{\phi(\theta_0,f),\dots,\phi(\theta_i,f)\}$ to build a new point $\theta_{i+1}$. For $n\in\NN$ we denote by $\mathcal M_n$ the class of all such methods that are allowed to make $n$ queries. As done by \citet{agarwal2010information}, we denote the error of the method $M$ on the function $f$ after $n$ steps as 
\[
\epsilon_n(M,f,\mathcal{C},\phi)= f(\theta_n)-\min_{\theta\in\mathcal C} f(\theta).
\]
Given a class of functions $\mathcal{S}$, an oracle $\phi$ and a convex constraint set $\mathcal{C}$, \citet{agarwal2010information} also  defines the minimax error as
\[
 \epsilon_n^*(\mathcal S,\mathcal{C},\phi)= \inf_{M\in \mathcal M_n} \sup_{f\in\mathcal S} \EE_{\phi} \epsilon_n(M,f,\mathcal{C},\phi).
\]

We will lower bound this minimax error by relating convex stochastic approximation with convex aggregation of estimators \citep{JudNem00,tsybakov2003optimal}.

\subsection{Convex aggregation of estimators }
Let $(\mathcal{X},\mathcal{A})$ be a measurable space and $\mathcal Y\subset \RR$ . We consider random variables $(X,Y)$ on $\mathcal{X}\times \mathcal{Y}$ with probability distribution denoted by $\pi$. We observe $n$ i.i.d. pairs $D_n=\{(X_1,Y_1),\dots, (X_n,Y_n)\}$ which follow the law $\pi$ and we want to predict the output $Y$ for any feature $X\in\mathcal{X}$ by a prediction $f(X)$ for a measurable function $f$ from $\mathcal{X}$ to $\RR$. For this purpose we want to minimize the risk defined by
\[
 A(f)=\EE [\ell(f(X),Y)],
\]
for any measurable function $f$ from $\mathcal{X}$ to $\RR$ and $\ell: \mathcal Y \times \mathcal Y\to \RR$ a loss function.

We consider we have access to $d$ different arbitrary estimators $\mathcal{F}=\{f_1,\dots, f_d\}$ with values in $\mathcal Y$. We denote their convex hull by $\mathcal{C}=conv(f_1,\dots,f_d)$. The aim of convex aggregation is to build a new estimator  which is a convex combination of the different $f_i$ and behaves as the best among the estimators  $f_i$.
The aggregation problem is equivalent to a minimization problem over the simplex $\Delta_d$ since for $f\in\mathcal{C}$ there is $\theta\in\Delta_d$ such that $f=\sum_{i=1}^d\theta(i) f_i$. Therefore, defining $B(\theta)=A\Big(\sum_{i=1}^d\theta(i) f_i\Big)$, we have  
\[
 \min_{f\in\mathcal{C}} A(f)= \min_{\theta\in\Delta_d} B(\theta).
\]
We denote by $F:\mathcal X \to \RR^d, x\mapsto (f_1(x),\dots, f_d(x)) $ the function whose the $i$th coordinate is the function $f_i$, and we have
\[
 B(\theta)= \EE [\ell (\langle F(X), \theta\rangle,Y)].
\]
Therefore the convex aggregation problem of minimizing $A(f)$ over the convex hull of $\mathcal F$ is formally equivalent to the stochastic approximation problem of minimizing, over the simplex $\Delta_d$, the function $B(\theta)=\EE [\ell (\langle F(X), \theta\rangle,Y)]$, given only unbiased estimates of its gradient $\nabla B_n(\theta)=\nabla \ell(\langle F(x_n),\theta\rangle,y_n)$. Hence lower bounds on convex aggregation problems provide lower bounds on stochastic approximation problems studied in this paper.

\subsection{Aggregation in regression and application to oracle complexity of stochastic quadratic optimization}

We first consider the regression problem for which $\mathcal{Y}=\RR$. We rely substantially on \citet{tsybakov2003optimal}. The regression model is 
\[
 Y_i=f_*(X_i)+\xi_i, \text{ for } i=1,\dots, n,
\]
where $X_1,\dots,X_n$ are i.i.d.~random vectors of $\mathcal{X}$ of law $P^X$ and $\xi_i$ are i.i.d.~ Gaussian $\mathcal{N}(0,\sigma^2)$ random variables such that $(\xi_1,\dots, \xi_n)$ is independent of $(X_1,\dots, X_n)$ and $f_*:\mathcal X\to \RR$ is the regression function. Regression problem aims to estimate the unknown regression function $f_*$ based on the data $D_n$ by minimizing the risk
\[
 A_{\text{reg}}(f,f_*)=\EE (f(X)-f_*(X))^2.
\]
The problem of the optimal rate of convex aggregation has been studied by \citet{tsybakov2003optimal}. We reintroduce his notations and assumptions for sake of completeness.

Let denote by $\mathcal{F}_0=\{f: \Vert f\Vert_\infty \leq L\}$ for $L>0$ and assume that 
\begin{description}
\item[\textbf{(B1)}]
There exists a cube $S\subset \mathcal X$ such that $P^X$ admits a bounded density $\mu$ on $S$ w.r.t. the Lebesgue measure and $\mu(x)\geq\mu_0>0$ for all $x\in S$.
\item[\textbf{(B2)}]
There exists a constant $c_0$ such that $ d \leq c_0 \exp(n)$.
\end{description}
We have the following result
\begin{theorem}[Theorem 2, \citet{tsybakov2003optimal}]\label{prop:aggreg}
 Under assumptions \textbf{(B1-2)} we have
 \[
  \sup_{f_1,\dots,f_d\in\mathcal F_0}\inf_{T_n} \sup_{f_*\in\mathcal{F}_0}\big[ \EE_{D_n} A_{\text{reg}}(T_n,f_*)-\min_{f\in\mathcal C} A_{\text{reg}}(f,f_*)\big]\geq c \zeta_n(d),
 \]
for some constant $c>0$ and any integer $n$, where $\inf_{T_n}$ denotes the infimum over all estimators, $\EE_{D_n}$ denotes the expectation with regard to the probability distribution of the data $D_n$ and
 \[
    \zeta_n(d)=\begin{cases}
                d/n & \text{ if } d\leq \sqrt{n} \\
                \sqrt{\frac{1}{n}\log\big( \frac{d}{\sqrt n} +1\big)} & \text{ if } d> \sqrt{n}.
               \end{cases}
  \]
\end{theorem}

We relate now the problem of convex aggregation of regression functions to the problem of stochastic quadratic functions optimization.
Consider $\mathcal F=\{f_{1},\dots, f_{d}\}$ the set of estimators given by Proposition~\ref{prop:aggreg} and denote by $F:\mathcal X \to \RR^d, x\mapsto (f_1(x),\dots, f_d(x)) $.
For $f\in \mathcal C$, there is  $\theta\in\Delta_d$ such that $f=\sum_{i=1}^d\theta(i) f_i$ and we obtain 
\[
 A_{\text{reg}}(f,f_*)=\EE[(\langle \theta,F(X)\rangle -f_*(X))^2]= B(\theta),
\]
where $B(\theta)=\langle \theta, \EE[ F(X)\otimes F(X)],\theta\rangle-2\langle \theta \EE[f_*(X)F(X)]\rangle +\EE[f_*(X)^2]$ is a quadratic function. This set enables us to construct a difficult subclass of quadratic functions:
\[
 \mathcal{G}_{\text{quad}}=\Big\{ B(\theta)=\frac{1}{2}\EE[(\langle \theta,F(X)\rangle -f_*(X))^2]; f_*\in\mathcal F_0\Big\}.
\]
We also define the first-order stochastic oracle $\phi_{\text{quad}}$ on   $\mathcal{G}_{\text{quad}}$ as follows
\[
 \phi_{\text{quad}}(\theta,f)=\bigg(\frac{1}{2}(\langle \theta,F(x) \rangle -f_*(x))^2, (\langle \theta,F(x) \rangle -f_*(x))F(x)\bigg),\text{ for } x\sim P^X.
\]
We can optimize $B$ with a stochastic approximation algorithm $M\in\mathcal M_n$ to obtain  $\theta_n\in\Delta_d$ and therefore build a estimator $T_n=\sum_{i=1}^d \theta_n(i)f_i$ which belongs to $\mathcal C$. Moreover we have
\[
 A_{\text{reg}}(T_n,f_*)= B(\theta_n)\text{ and } \min_{f\in\mathcal C} A_{\text{reg}}(f,f_*)=\min_{\theta\in\Delta_d}B(\theta).
\]
Consequently, for the oracle $\phi_{\text{quad}}$ and the class $ \mathcal{G}_{\text{quad}}$ Proposition~\ref{prop:aggreg} implies that 
\begin{equation}
  \epsilon_n^*(\mathcal{G}_{\text{quad}},\Delta_d,\phi_{\text{quad}}) \geq  c \zeta_n(d).
\end{equation}
And we have proven the following minimax oracle complexity.
\begin{proposition}\label{prop:oraclequad}
 Let $\Delta_d$ be the simplex. Then there exists universal constants $c_0>0$ and $c>0$ such that the minimax oracle complexity over the class $\mathcal{S}_\text{quad}$ of quadratic functions satisfies the following lower bounds:
 \begin{itemize}
  \item For $d\leq \sqrt{n}$
  \[
   \sup_{\phi\in\Phi} \epsilon_n^*(\mathcal{S}_\text{quad},\Delta_d,\phi) \geq c \frac{d}{n}.
  \]
 \item For $ \sqrt{n}\leq d \leq c_0\exp(n)$
  \[
   \sup_{\phi\in\Phi} \epsilon_n^*(\mathcal{S}_\text{quad},\Delta_d,\phi) \geq c   \sqrt{\frac{1}{n}\log\Big( \frac{d}{\sqrt n} +1\Big)}.
  \]
 \end{itemize}
\end{proposition}

We note that without assumption on $d$ the lower-bound for the class of  quadratic functions is of order $O(1/n)$ but in high-dimensional settings it becomes of order $(1/\sqrt{n})$. Nevertheless we will see in the next section this lower-bound is always of order $(1/\sqrt{n})$ for the class of linear functions. 

\subsection{Aggregation in classification and application to oracle complexity of stochastic linear optimization}

We consider now the classification problem with the hinge loss for which $\mathcal{Y}=\{-1,1\}$. We follow very closely the framework of \citet{Lec06,Lec07} and use their notations. We still consider random variables $(X,Y)$ on $\mathcal{X}\times \mathcal{Y}$ with probability distribution denoted by $\pi$. We observe $n$ i.i.d. pairs $D_n=\{(X_1,Y_1),\dots, (X_n,Y_n)\}$ which follow the law $\pi$ and we want to predict the label $Y$ for any feature $X\in\mathcal{X}$ by minimizing the hinge risk defined by
\[
 A_\text{cla}(f)=\EE\max(1-Yf(X),0),
\]
for any measurable function $f$ from $\mathcal{X}$ to $\RR$.
 We consider we have access to $d$ different estimators $\mathcal{F}=\{f_1,\dots, f_d\}$ with values in $[-1,1 ]$. We denote their convex hull by $\mathcal{C}=conv(f_1,\dots,f_d)$.
\citet[Theorem 1]{Lec06} and \citet[Theorem 2]{Lec07} provide a lower bound on  this aggregation problem for classification we adapt to our specific case.
\begin{proposition}[ Adaptation of Theorem 2  of \citet{Lec07}\label{theo:lecu} for $\kappa=\infty$]\label{prop:lowerlin}
Let $d,n$ be two integers such that $2\log_2 d\leq n $. We assume that the input space $\mathcal{X}$ is infinite. There exists an absolute constant $c>0$, and a set of prediction rules $\mathcal{F}=\{f_1,\dots,f_n\}$ such that for any real-valued procedure $T_n$, there exists a probability measure $\pi$, for which
\[
  \EE_{D_n}[A_\text{cla}(T_n)]- \min_{f\in\mathcal{C}}(A_\text{cla}(f))\geq c \sqrt{\frac{\log d}{n}}.
\]
 \end{proposition}
\begin{proof}
Theorem 2 of \citet{Lec07} is stated under an additional Margin assumption MAH$(\kappa)$ (see definition and notation below Eq. (9) in \citet{Lec07}) on the probability distribution $\pi$, i.e., there exists a constant $c_0$ such that 
\[
 \EE [\vert f(X)-f^*(X)\vert]\leq c_0 (A(f)-A^*)^{1/\kappa},
\]
for any function $f$ on $\mathcal X$ with values in $[-1,1]$. Therefore taking $\kappa \to \infty$, we can always consider $c_0=2$. And the constant $c(\kappa)$ in Theorem 2 of \citet{Lec07} is 
\[ 
c(\kappa)=c_0^{\kappa}(4e)^{-1}2^{-2\kappa(\kappa-1)/(2\kappa-1)}(\log 2)^{-\kappa/(2\kappa-1)},
\]
which goes when $\kappa\to\infty$ to $c_\infty=\sqrt{2}/(4e\sqrt{\log2})$. Hence taking $\kappa\to \infty$ in Theorem 2 of \citet{Lec07} implies  Proposition~\ref{theo:lecu}. We could also have plugged arguments of the proof of Theorem 14.5 of \citet{DevGyoLug96} to directly prove this result. 
\end{proof}
We relate now the problem of convex aggregation of classifiers to the problem of optimizing a linear function on the simplex. Consider the set of prediction rules $\mathcal{F}=\{f_1,\dots,f_n\}$ given by Proposition~\ref{prop:lowerlin} and denote by $F:\mathcal X \to \RR^d, x\mapsto (f_1(x),\dots, f_d(x)) $.
For $f\in \mathcal C$, there is  $\theta\in\Delta_d$ such that $f=\sum_{i=1}^d\theta(i) f_i$ and we obtain
 \[
 A_\text{cla}(f)=  \EE\max(1-Y\langle F(X), \theta\rangle,0).
 \]
 On the other hand, when the $f_i$ are valued in $[-1,1]$, the classification problem becomes equivalent to maximize the expectation $\EE Yf(X)$ since the hinge loss is linear on $[-1,1]$:
 \[
   Y\in\{-1,1\}, f(X)\in[-1,1] \implies Yf(X)\in[-1,1] \implies \EE\max(1-Yf(X),0)=1-\EE Yf(X).
 \]
 Combining both, we obtain that 
 \[
 A_\text{cla}(f)=1-\langle \EE [YF(X)], \theta\rangle=1+C(\theta),
 \] 
 where $C(\theta)=-\langle\EE[ YF(X)], \theta\rangle$ is a linear function. This set enables us to construct a difficult subclass of linear functions
\[
 \mathcal{G}_{\text{lin}}=\{ C(\theta)=-\langle\EE[ YF(X)], \theta\rangle; (X,Y)\sim\pi\}.
\]
We also define the first-order stochastic oracle $\phi_\text{lin}$ on   $\mathcal{G}_{\text{lin}}$ as follows
\[
 \phi_\text{lin}(\theta,f)=\Big(\langle yF(x), \theta\rangle,  yF(x)\Big),\text{ for } (x,y)\sim \pi .
\]
As before we may optimize $C$ with a stochastic approximation algorithm $M\in\mathcal M_n$ to obtain  $\theta_n\in\Delta_d$ and therefore build a estimator $T_n=\sum_{i=1}^d \theta_n(i)f_i$ which belongs to $\mathcal C$. Moreover we have
\[
 A_\text{cla}(T_n)= C(\theta_n)\text{ and } \min_{f\in\mathcal C} A_\text{cla}(f)=\min_{\theta\in\Delta_d}C(\theta).
\]
Consequently, for the oracle $\phi_\text{lin}$ and the class $ \mathcal{G}_{\text{lin}}$ Proposition~\ref{prop:aggreg} implies that 
\begin{equation}
  \epsilon_n^*(\mathcal G_\text{lin},\Delta_d,\phi_\text{lin}) \geq  c \sqrt{\frac{\log d}{n}} .
\end{equation}
And we have proven the following minimax oracle complexity.
\begin{proposition}\label{prop:oraclelin}
 Let $\Delta_d$ be the simplex. Then there exists universal constant $c>0$ such that the minimax oracle complexity over the class $\mathcal{S}_\text{lin}$ of linear functions satisfies the following lower bound for $2\log_2 d\leq n$ 
  \[
   \sup_{\phi\in\Phi} \epsilon_n^*(\mathcal{S}_\text{lin},\Delta_d,\phi) \geq c \sqrt{\frac{\log(d)}{n}}.
  \]
\end{proposition}

\section{Lower-bound on the rates of convergence of DA and MD algorithms}\label{app:compa}
Let us consider in this section that $f=0$, $g(\theta)=\frac{1}{2\nu}\Vert \theta-\theta_*\Vert_2^2$ and $h=\frac{1}{2}\Vert \theta\Vert_2^2$. In this case, for $n\geq1$, MD iterates $(\theta^{\text{md}}_{n})$ verify 
\[
\theta^{\text{md}}_{n}=\argmin_{\theta \in {\RR^d}} \Big\{ \frac{1}{2\nu}\Vert \theta-\theta_*\Vert_2^2 +\frac{1}{2\gamma} \Vert \theta-\theta^{\text{md}}_{n-1}\Vert_{2}^{2}\Big \}.
\]
Therefore $\theta^{\text{md}}_{n}=\theta_*+\frac{1}{\gamma \nu} (\theta^{\text{md}}_{n-1}-\theta_*)$, $\theta^{\text{md}}_{n}-\theta_*=\frac{1}{(\gamma \nu)^n} (\theta^{\text{md}}_{0}-\theta_*)$ and 
\[
g(\theta^{\text{md}}_{n})-g(\theta_*)=\frac{g(\theta^{\text{md}}_{0})-g(\theta_*)}{(\gamma \nu)^{2n}}.\]
Whereas DA iterates $(\theta^{\text{da}}_{n})$ satisfy
\[
\theta^{\text{da}}_{n}=\argmin_{\theta \in {\RR^d}} \Big\{ \frac{1}{2\nu}\Vert \theta-\theta_*\Vert_2^2 +\frac{1}{2\gamma n } \Vert \theta \Vert_{2}^{2}\Big \}.
\]
We compute $\theta^{\text{da}}_{n}=\frac{\gamma \nu n}{\gamma \nu n+1}\theta_*$ and 
\[
g(\theta^{\text{da}}_{n})-g(\theta_*)=\frac{g(\theta^{\text{da}}_{0})-g(\theta_*)}{(\gamma n)^2}.
\]

\section{Continuous time interpretation of DA et MD}\label{app:continuousint}
Following \citet{NemYud83, Kri15} we propose a continuous interpretation of these methods  for $g$ twice differentiable. We note this could be extended for $g$ non-smooth with differential inclusions.
\paragraph{Derivation of the  ordinary differential equation (ODE).}

 The first-order optimality condition  of the MD iteration in  \eq{mdprox}
$
 \gamma \nabla f(\theta_n)+\gamma \nabla g(\theta_{n+1}) +\nabla h(\theta_{n+1})-\nabla h(\theta_n)
$ can be rearranged as 
\[
 \frac{\nabla h(\theta_{n+1})-\nabla h(\theta_n)}{\gamma}=- \nabla f(\theta_n)- \nabla g(\theta_{n+1}).
\]
Noting $\partial_t  \nabla h(\theta)=\nabla^2h(\theta) \dot\theta$, this is exactly a forward-backward Euler discretization of the MD ODE  
\begin{equation}\label{eq:odemd}
\dot \theta = -\nabla^2h(\theta)^{-1}[\nabla f(\theta)+\nabla g(\theta)].
\end{equation}
On the other hand, considering the DA iteration in \eq{da} we obtain 
\begin{equation}\label{eq:da1}
\frac{\eta_{n}-\eta_{n-1}}{\gamma}
=
-\nabla f (\theta_{n-1}) \qquad \text{ and } \qquad
\eta_n=n\gamma \nabla g(\theta_n)+\nabla h(\theta_n).
\end{equation}
Combining both parts in \eq{da1} leads to the single equation 
\[
n\gamma \frac{\nabla  g(\theta_n)-\nabla g(\theta_{n-1})}{\gamma} + \nabla g(\theta_{n-1})+\frac{\nabla h(\theta_n)-\nabla h(\theta_{n-1})}{\gamma}=-\nabla f(\theta_{n-1}),
\]
which is the explicit Euler discretization of the ODE
$
 \partial_t(t\nabla g(\theta)+\nabla h(\theta))=-\nabla f(\theta).
$
Therefore the ODE associated to DA takes the form
\begin{equation}\label{eq:odeda}
 \dot \theta= -\nabla^2(h(\theta)+t g(\theta))^{-1}(\nabla f(\theta)+\nabla g(\theta)).
\end{equation}
It is worth noting that this ODE is very similar to the MD ODE in \eq{odemd}, with an additional term $t g(\theta)$ in the inverse mapping $\nabla^2(h(\theta)+t g(\theta))^{-1}$ which may thus slow down the DA dynamic.

\paragraph{Lyapunov analyzes.}
Lyapunov functions are used to prove convergence of the solutions of ODEs. In analogy with the discrete case, the Bregman divergence is a Lyapunov function for these ODEs \citep[see, e.g.,][]{Kri15} since
\BEAS
\partial_t D_h(\theta_*,\theta(t))
&=&\partial_t[h(\theta_*)-h(\theta(t))-\langle \nabla h(\theta(t)),\theta_*-\theta(t)\rangle]\\
&=& -\langle \nabla h(\theta(t)),\dot \theta(t)\rangle  +\langle \nabla^2 h(\theta(t))\dot \theta(t),\theta(t)-\theta_*\rangle +\langle \nabla h(\theta(t)),\dot \theta(t)\rangle\\
&=& \langle \nabla^2 h(\theta(t))\dot \theta(t),\theta(t)-\theta_*\rangle.
\EEAS
For the MD ODE  in \eq{odemd} we obtain
\BEAS
\partial_t D_h(\theta_*,\theta(t))
&=& -\langle \nabla f(\theta(t))+\nabla g(\theta(t)),\theta(t)-\theta_*\rangle\\
&\leq& \psi(\theta_*)-\psi(\theta(t))\qquad \text{(by convexity of $\psi$)}.
\EEAS
Integrating, this yields with Jensen inequality 
\[
 \psi(\bar\theta(t))-\psi(\theta_*)\leq \frac{1}{t}\int_0^t \big(\psi(\theta(s))-\psi(\theta_*)\big)ds\leq\frac{ D_h(\theta_*,\theta(0))-D_h(\theta_*,\theta(t))}{t},
\]
for $\bar\theta(t)=\frac{1}{t}\int_0^t \theta(s)ds$. This is the same convergence result as in the discrete time. For the DA ODE  in \eq{odeda} we obtain
 \BEAS
\partial_t D_{h+tg}(\theta_*,\theta(t))
&=&\partial_t[(h+tg)(\theta_*)-(h+tg)(\theta(t))-\langle \nabla (h+tg)(\theta(t)),\theta_*-\theta(t)\rangle]\\
&=&g(\theta_*) -\langle (\nabla h(\theta(t))+t\nabla g(\theta(t))),\dot \theta(t)\rangle +g(\theta(t))\\
&&+\langle \partial_t(\nabla h+t\nabla g)(\theta(t)),\theta(t)-\theta_*\rangle +\langle (\nabla+t\nabla g) h(\theta(t)),\dot \theta(t)\rangle\\
&=& g(\theta_*)-g(\theta_t) -\langle \nabla f(\theta(t)),\theta(t)-\theta_*\rangle.
\EEAS
Therefore by convexity of $f$,
$
\partial_t D_{h+tg}(\theta_*,\theta(t))\leq  \psi(\theta_*)-\psi(\theta(t))
$
and we obtain 
\[
 \psi(\bar\theta(t))-\psi(\theta_*)\leq\frac{ D_{h}(\theta_*,\theta(0))-D_{h
+tg}(\theta_*,\theta(t))}{t}.
\]
The continuous time argument  really mimics the proof of Proposition~\ref{prop:phifunction} without the technicalities associated with the  discrete time.  We remind that we recover the variational interpretation of \citet{Kri15,WibWilJor16,WilRecJor16}: the Lyapunov function generates the dynamic in the sense that a function $L$ is first chosen and secondly a dynamics, for which $L$ is a Lyapunov function, is then designed.  In this way MD and DA are the two different dynamics associated to the two different Lyapunov functions $D_h$ and $D_{h+tg}$.

\paragraph{Extension to the noisy-gradient case.}
We consider now we only have access to noisy estimates of the gradient as in \mysec{stogen} and propose a continuous-time interpretation of these stochastic methods.
Stochastic MD and SDA may be viewed, in their primal-dual forms, as discretizations of the following stochastic differential equations (SDE). For stochastic MD
\[
  \mathrm{d} \eta(t) = -[\nabla f(\theta(t)) +\nabla g(\theta(t))] \mathrm{d}t +\sigma  \mathrm{d}W(t)\mathrm{d}t \qquad \text{ and } \qquad \eta(t)=\nabla h(\theta(t)),
\]
and for SDA
\[
 \mathrm{d} \eta(t) = -\nabla f(\theta(t)) \mathrm{d}t +\sigma  \mathrm{d}W(t)\mathrm{d}t  \qquad \text{ and } \qquad \eta(t)=\nabla (h+tg)(\theta(t)),
\]
where $W_t$ is a Wiener process and $\sigma>0$. We note that the regularization $g$ does not take part in the SDA SDE which explains this dynamic is efficient in presence of noise. In contrast, the stochastic MD SDE is corrupted by the presence of the gradient $\nabla g$ which may not behaves well for non-smooth $g$. This continuous-time interpretation of stochastic algorithms could lead to further insights but is outside the scope of this paper.

\section{Examples of different geometries}\label{app:examples}

We describe now different examples of concrete geometries and how SDA is then implemented for well known regularizations $g$.

\paragraph{Euclidean distance.}
The simplest geometry is obtained by taking the function $h(\theta)=\frac{1}{2}\Vert \theta\Vert_2^2$, which is a Legendre function on $\dom h ={\RR^d}$. Its associated Bregman divergence is also the squared Euclidean distance $D_h(\alpha,\beta)=\frac{1}{2}\Vert \alpha-\beta\Vert_2^2$. Therefore \textbf{(LC)} is equivalent to the smoothness of the function $f$ and we return to classic results on proximal gradient descent.
\begin{itemize}
\item Projection: Let $g=\mathbbm 1_\mathcal{C}$ be the indicator of a convex set $\mathcal{C}$. The SDA method yields to the projected method
\[
\theta_n=\min_{\theta\in \mathcal{C} }\Big \Vert \theta+\gamma\sum_{k=0}^{n-1}\nabla f_{k+1}(\theta_k)\Big\Vert_2^2.
\]

 \item
 $\ell_2$-regularization: 
Let $g=\frac{1}{2} \Vert\cdot\Vert_Q^2$  where $Q\succcurlyeq0$, we directly have  $\nabla h^*_n(\eta) =(\idm +n\gamma Q)^{-1}\eta$ and the SDA method comes back to 
\[
 \theta_n=\theta_{n-1}-(\gamma^{-1}\idm+nQ)^{-1}(Q\theta_{n-1}+\nabla f_{n}(\theta_{n-1})), \text{ for } n\geq1,
\]
which is a standard gradient descent on $f+g$ with a structured decreasing step-size $\gamma_n= (\gamma^{-1}\idm+nQ)^{-1}$.

\item $\ell_1$-regularization:
Let  $g=\lambda \Vert\cdot\Vert_1$, we can compute the primal iterate with, for $i=1,\dots,d$, $\nabla_ih^*_n(\eta)=\sign(\eta(i))\max(\vert\eta(i)\vert-n\gamma\lambda,0)$ . Therefore the SDA method is equivalent to the iteration:
\[
 \theta_n(i)=-\sign\Big(\sum_{k=0}^{n-1}\nabla_if_{k+1}(\theta_k)\Big)\max\bigg(\Big\vert\sum_{k=0}^{n-1}\nabla_if_{k+1}(\theta_k)\Big\vert-n\gamma\lambda,0\bigg)\text{ for } i=1,\dots,d.
\]
Yet since convergence results hold on the average of the iterates $\bar \theta_n$, SDA provides less sparse solutions than other methods which rather consider final iterates as outputs.
\end{itemize}

\paragraph{Kullback-Leibler divergence.}
The negative entropy $h(\theta)=\sum_{i=1}^n \theta(i) \log(\theta(i))$ is a Legendre function on  $\dom h=(0,\infty)^n$ whose associated  Bregman divergence is  the Kullback-Leibler divergence 
\[
D_h(\alpha,\beta)=\sum_{i=1}^n \alpha(i) \log\Big(\frac{\alpha(i)}{\beta(i)}\Big)+\sum_{i=1}^n (\beta(i)-\alpha(i)),
\]
and its conjugate gradient mapping is $\nabla_i h^*(\eta)=\exp(\eta_i)$ for ${i=1,\dots,d}$.

Since $h$ is $1$-strongly convex with respect to the $\ell_{1}$-norm \citep[see, e.g.,][Proposition 5.1]{Teb03}, \textbf{(LC)} holds, for example, if $f$ is smooth with regards to the $\ell_{1}$-norm. This illustrates one of the non-Euclidean benefit since Lipschitz constants under the  $\ell_{\infty}$-norm are  smaller than under the $\ell_{2}$-norm.

This geometry is particularly appropriated to constrained minimization on the simplex $\Delta_d$. With $g(\theta)=\mathbbm{1}_{\Delta_d}$, SDA update is the dual averaging analogue of the exponentiated gradient algorithm \citep{KivWar97}:
\[
 \theta_n(i)=\frac{\exp (\eta_n(i))}{\sum_{j=1}^d\exp(\eta_n(j))} \text{ for } i=1,\dots,d.
\]

\paragraph{$\ell_p$-norm.}

The choice $h=\frac{1}{2(p-1)}\Vert \cdot\Vert_p^2$ for $p\in(1,2]$ is believed to adapt to the geometry of learning problem and is often used with $p=1+1/\log(d)$ in association with $\ell_1$-regularization \citep[see, e.g.,][]{Duc10}. Its Fenchel conjugate is the squared conjugate norm $h^*=\frac{1}{2(q-1)}\Vert \cdot\Vert_q^2$ for $1/p+1/q=1$ and its conjugate gradient mapping is $\nabla_ih^*(\eta)=\frac{\sign (\eta(i))\vert \eta(i)\vert^{q-1}}{(q-1)\Vert \eta\Vert_q^{q-2}} $ \citep[see, e.g.,][]{GenLit99}.
For $\ell_1$-regularization, this yields to:
\[
 \nabla_i h_n^*(\eta)=\nabla_i h^*\big(\sign(\eta(i))\max(\vert\eta(i)\vert-n\gamma\lambda,0)\big) \text{ for } i=1,\dots,d.
\]

The function $h$ is $1$-strongly convex with respect to the $\ell_p$-norm \citep[see, e.g.,][]{Han56}. Therefore \textbf{(LC)} holds if $f$ is smooth with respect to the the $\ell_{p}$-norm. However when  the function $f$ considered is quadratic as in \mysec{stogen},
 we can directly show that \textbf{(LC)} holds under tighter conditions on the Hessian matrix $\Sigma$ (see proof in \myapp{proofp}). 
\begin{proposition}\label{prop:pdivergence}
Assume that $f(\theta)=\frac{1}{2}\langle \theta, \Sigma\theta\rangle$ and $h(\theta)=\frac{1}{2(p-1)}\Vert \theta \Vert_p^2$.
 Then $h-\gamma f$ is convex for any constant step-size $\gamma$ such that 
 \[
 \gamma \leq \min_\alpha \frac{\Vert \alpha\Vert _p^2}{\langle \alpha,\Sigma \alpha\rangle}.
 \]
\end{proposition}
When $\Sigma=\EE (x\otimes x)$ is a covariance matrix as in \mysec{leastsquares}, $\langle \alpha,\Sigma \alpha\rangle=\EE \langle x,\alpha\rangle^2\leq \EE \Vert x\Vert_q^2\Vert \alpha\Vert_p^2$ by H\"older inequality,
and  Proposition~\ref{prop:pdivergence} admits the following corollary.
\begin{corollary}\label{cor:pdivergence}
 Assume that $f(\theta)=\frac{1}{2} \EE(\langle x,\theta\rangle-y)^2$, $h(\theta)=\frac{1}{2}\Vert \theta \Vert_p^2$ and $q$ such that $1/p+1/q=1$.
 Then $h-\gamma f$ is convex for any constant step-size $\gamma$ such that 
 \[
 \gamma \leq 1/\EE \Vert x \Vert_q^2.
 \]
\end{corollary}
Therefore we may use the algorithm with bigger step-size than in the Euclidean case. Moreover when the algorithm is started from $\theta_0=0$, the Bregman divergence is  $D_h(\theta_*,\theta_0)=\frac{1}{2(p-1)}\Vert \theta_* \Vert_p^2$ and the bias in Proposition~\ref{prop:phifunction} would be bounded by $\frac{\EE \Vert x \Vert_q^2\Vert \theta_* \Vert_p^2}{2(p-1)}$.

For high-dimension problems, taking $q=1+\log(d)$ (with $p \sim 1$ and $q\sim +\infty$) yields to bounds depending on the $\ell_1$-norm of the optimal predictor and the $\ell_\infty$-norm of the features which is advisable for sparse problems.

 \section{Proof of Proposition~\ref{prop:pdivergence}}\label{app:proofp}
We consider here $h(\theta)=\frac{1}{2(p-1)}\Vert \theta\Vert_p^2$. 
For $\theta\in{\RR^d}$, $h$ is twice differentiable. Its gradient is 
\[
\nabla_ih(\theta)=\frac{\sign(\theta(i))\vert \theta(i)\vert^{p-1}}{(p-1)\Vert \theta \Vert _p^{p-2}},
\]
and its Hessian may be written for  $\alpha= \frac{2-p}{(p-1)}\Vert \theta\Vert_p^{-2(p-1)}$,  $u(i)= \Vert \theta\Vert_p^{2-p}\theta(i)^{p-2}$ and $v(i)=\theta(i)^{p-1}$ for $i=1,\dots,d$, as
\[
\nabla^2h (\theta)= \Diag(u)+\alpha v v^\top,
\]
The function $h-\gamma f$ is convex  if and only if $\nabla^2 h(\theta)\preccurlyeq \gamma \Sigma$ for all $\theta\in{\RR^d}$. This condition is equivalent to
\[
 \min_\theta \min_\alpha \frac{\langle \alpha,\nabla^2h (\theta)\alpha\rangle}{\langle \alpha,\Sigma \alpha\rangle}\geq \gamma.
\]
A sufficient condition  is that $ \Diag u\succcurlyeq \gamma \Sigma$. After a change of variables, $u$ may be written as $u(i)=\eta(i)^{p-2}$ where $\eta(i)=\vert \theta(i)\vert /\Vert \theta\Vert_p$ satisfies $\sum_{i=1}^d\eta(i)^p=1$ and $\eta(i)\geq0$.
Hence for all $\theta,\alpha\in{\RR^d}$ 
\[
 \langle \alpha,\nabla^2h (\theta)\alpha\rangle\geq  \sum_{i=1}^d \alpha(i)^2u(i)=\sum_{i=1}^d \alpha(i)^2 \eta(i)^{p-2},
\]
which implies
\[
\min_{\theta\in{\RR^d}}     \langle \alpha,\nabla^2h (\theta)\alpha\rangle\geq \min_{\eta\in{\RR^d}} \sum_{i=1}^d \alpha(i)^2 \eta(i)^{p-2}\text{ such that } \sum_{i=1}^d\eta(i)^p=1 \text{ and  }\eta(i)\geq0.
    \]
This optimization problem is equivalent with $v(i)=\eta(i)^p$  to  the one the simplex $\Delta_d$
\[
  \min_{v\in{\RR^d}} \sum_{i=1}^d \alpha(i)^2 v(i)^{1-2/p}\text{ such that } \sum_{i=1}^d v(i)=1 \text{ and  }\nu(i)\geq0,
\]
for which we define the Lagrangian $\mathcal{L}(v,\lambda,\mu)= \sum_{i=1}^d \alpha(i)^2 v(i)^{1-2/p}-\langle \lambda,v\rangle+\nu (1- \sum_{i=1}^d v(i))$ for $\lambda\in\RR^d_+$ and $\mu\in\RR$. Its gradient is $\nabla _{v(i)} \mathcal{L}(v,\lambda,\mu) = (1-2/p)\alpha(i)^2/v(i)^{2/p}-\lambda(i)-\nu$. Writing the KKT  condition for this problem \citep[see, e.g.,][]{BoyVan04}, we have that $(v,\lambda,\nu)$ is optimal if and only if $ (1-2/p)\alpha(i)^2/v(i)^{2/p}-\lambda(i)-\nu=0$, $\sum_{i=1}^d v(i)=1$ and for all $i$; $\lambda(i)\geq0$, $v(i)\geq0$ and $\lambda(i)v(i)=0$. These conditions are satisfied by $v(i)=\frac{\alpha(i)^p}{\sum_{i=1}^d a(j)^p}$, $\alpha(i)=0$ and $\nu=(1-2p)( \sum_{i=1}^d a(j)^p)^{2/p}$. Hence the  minimum value is
\[
 \sum_{i=1}^d \alpha(i)^2 v(i)^{1-2/p}=\sum_{i=1}^d \alpha(i)^2 \frac{\alpha(i)^{p-2}}{(\sum_{i=1}^d a(j)^p)^{1-2/p}}=\frac{\sum_{i=1}^d a(j)}{(\sum_{i=1}^d a(j)^p)^{1-2/p}}=\Vert \alpha\Vert_p^2.
\]
Consequently 
\[
 \langle \alpha,\nabla^2h (\theta)\alpha\rangle\geq \Vert \alpha\Vert_p^2,
\]
and $h-\gamma f$ is convex for $\gamma \leq \min_{\alpha \in{\RR^d}} \frac{\Vert \alpha\Vert _p^2}{\langle \alpha,\Sigma \alpha\rangle}$.

\section{Standard benchmarks}\label{app:expsido}

We have considered the \emph{sido} dataset which is often used for comparing large-scale optimization algorithms. This is a \emph{finite} binary classification dataset with finite number of observations with outputs in $\{-1,1\}$. We have followed the following experimental protocol: 
(1) remove all outliers, i.e., sample points $x_n$ whose norms is greater than $5$ times the average norm. 
(2) divide the dataset in two equal parts, one for training, one for testing,
(3) start the algorithms from $\theta_0=0$,
(4) sample within the training dataset with replacement, for $100$ times the number of observations in the training set; a dashed line marks the first effective pass in all plots,
(5) compute averaged cost on training and testing data  based on $10$ replications.
All cost are shown in log-scale, normalized to that the first iteration leads to $\psi(\theta_0)-\psi(\theta_*)=1$.

We solved a $\ell_1$-regularized least-squares regression for three different values of $\ell_1$-regularization: (1) one with the $\lambda_*$  which corresponds to the best generalization error after $500$ effective passes through the train set, (2) one with $\lambda_*/8$ and (3) one with $256\lambda_*$.

We compare five algorithms: averaged SGD with constant step-size, average SGD with decreasing step-size $C/(R^2\sqrt{n})$,  SDA with constant step-size, SDA with decreasing step-size $C/(R^2\sqrt{n})$ and SAGA with constant step-size  \citep{DefLacBac}, which showed state-of-the-art performance in the set-up of finite data sets. We consider the theoretical value of step-size which ensures convergence. We note the behaviors are comparable to the situation where step-sizes with the best testing error after one effective pass through the data (testing powers of $4$ times the theoretical step-size) are used.

We can make the following observations:
\begin{itemize}
\item
 We show results for $\lambda=\lambda_*$ in \myfig{sido}. SAGA, constant-step-size SDA and constant-step-size SGD exhibit the best behavior for both settings of step-size. However the training error of  SGD does not converge to $0$. On the other hand, SGD and SDA with step-size decaying as $C/R^2\sqrt{n}$ are slower.  SAGA and constant-step-size SDA exhibit some overfitting after more than $10$ passes on the regularized objective $\psi$.
 \vspace*{-.2cm}
\item
 We show results for $\lambda=\lambda_*/8$ in \myfig{sidosmall}. The problem is then very little regularized and the behavior of  constant-step-size SGD gets closer to constant-step-size SDA. There is here still overfitting for the regularized objective $\psi$. 
\item
\vspace*{-.2cm}
 We show results for $\lambda=256\lambda_*$ in \myfig{sido}. The problem is then much more regularized. In this case the regularization has an important weight and the stochasticity of the quadratic objective plays a minor role. Therefore SAGA exhibits the best behavior, despite strong early oscillations, with a linear convergence but reaches a saturation point after few passes over the data. On the other hand, constant-step-size SDA exhibits a sublinear convergence which is faster at the beginning and catches up with SAGA at the end. Constant-step-size SGD is not converging to the solution.   
\end{itemize}
 To conclude, constant-step-size SDA behaves similarly to SAGA which is specially dedicated to the set-up of finite data sets. For larger datasets, where only a single pass is possible, SAGA could not be run. Moreover SAGA does not come with generalization guarantees while SDA does (if a single pass is made).  
\begin{figure}[!h]
\centering
\begin{minipage}[c]{.40\linewidth}
\includegraphics[width=\linewidth]{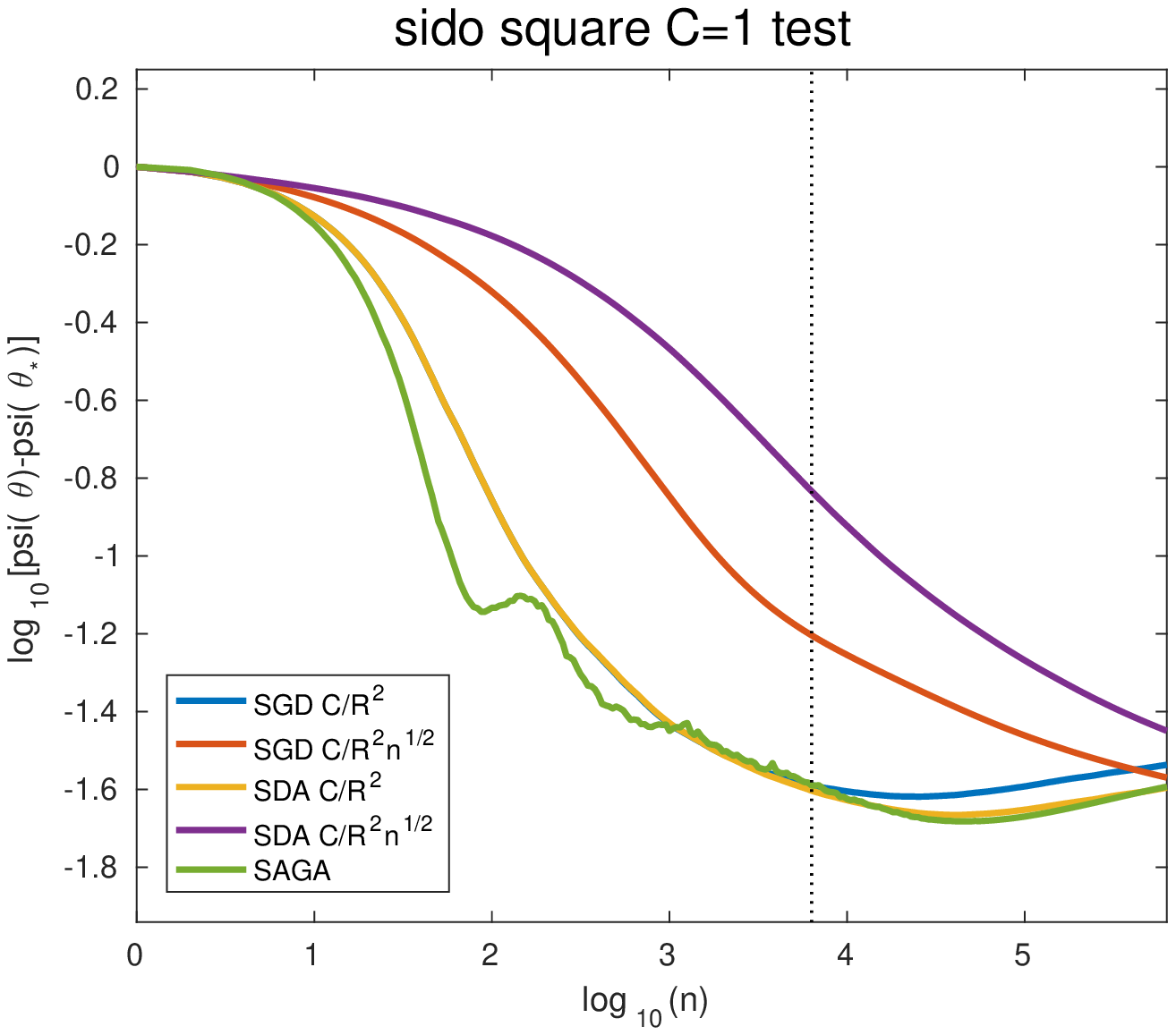}
   \end{minipage} \hspace*{.08\linewidth}
   \begin{minipage}[c]{.40\linewidth}
\includegraphics[width=\linewidth]{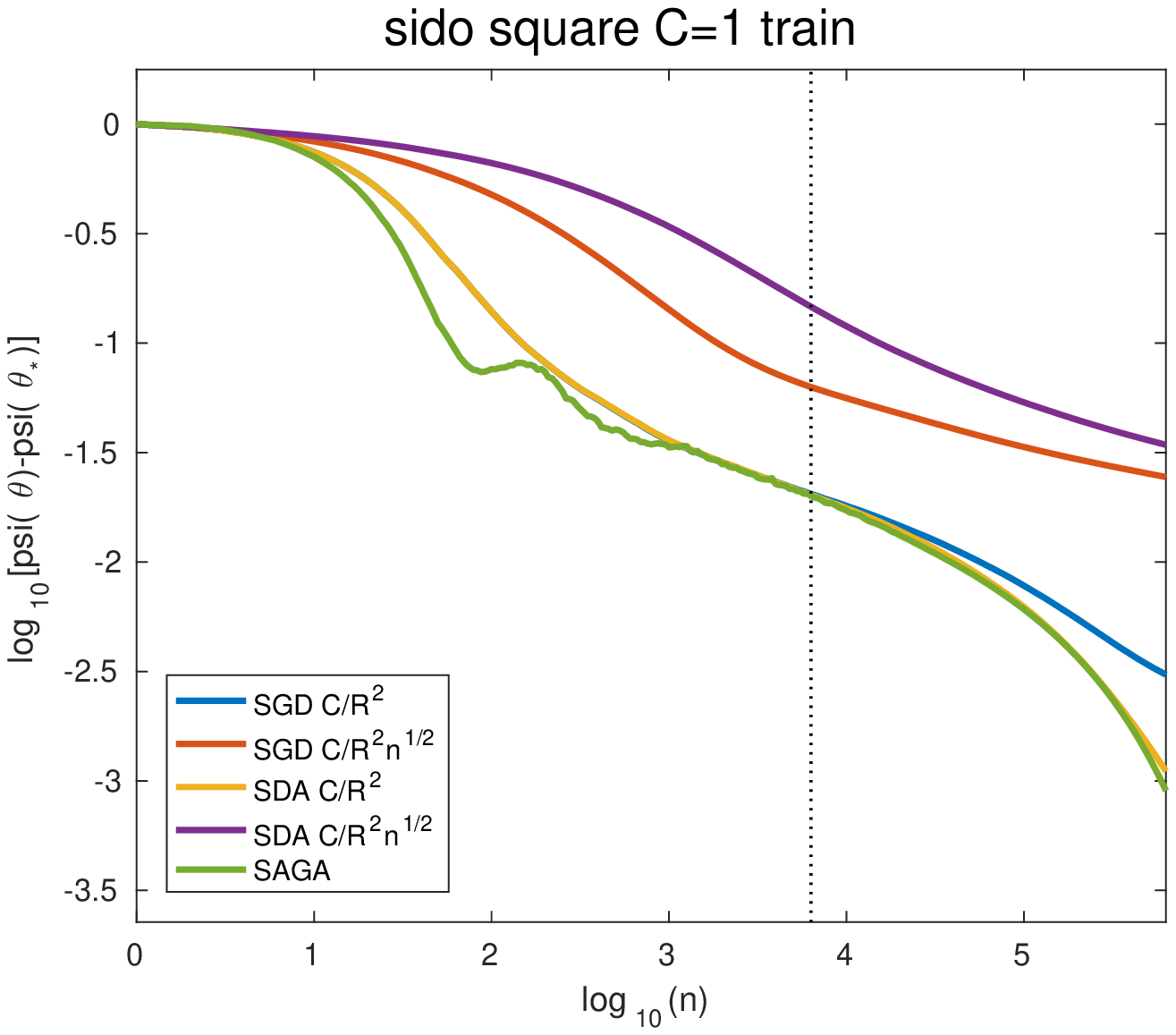}
   \end{minipage} 
  \caption{Test and train performances for $\ell_1$-regularized least-squares regression on the \emph{sido} dataset with $\lambda=\lambda_{\text{opt}}$. Left: test performance. Right: train performance.}
     \label{fig:sido}
     \begin{minipage}[c]{.40\linewidth}
\includegraphics[width=\linewidth]{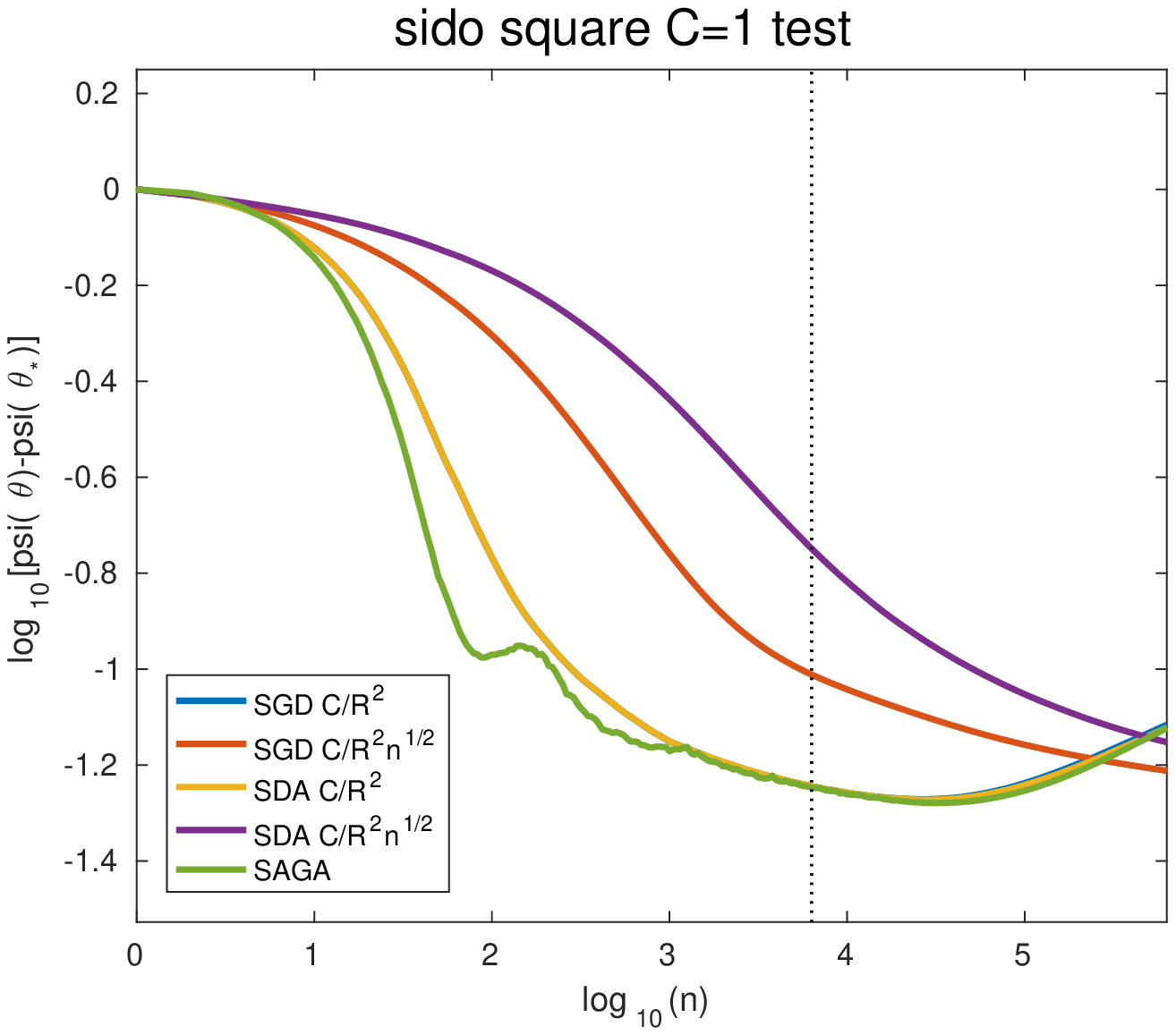}
   \end{minipage} \hspace*{.08\linewidth}
   \begin{minipage}[c]{.40\linewidth}
\includegraphics[width=\linewidth]{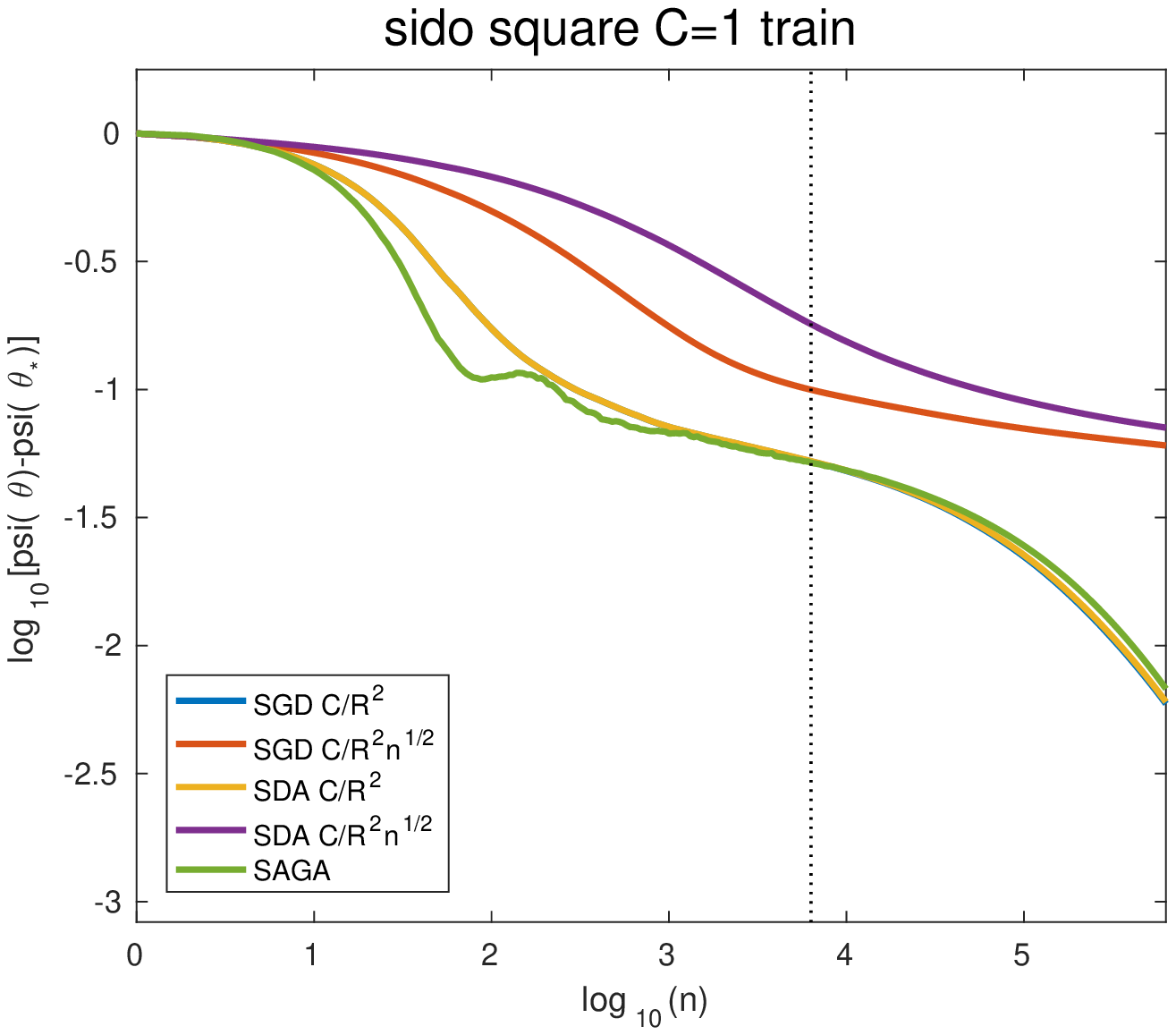}
   \end{minipage} 
 \caption{Test and train performances for $\ell_1$-regularized least-squares regression on the \emph{sido} dataset with $\lambda=\frac{\lambda_{\text{opt}}}{8}$. Left: test performance. Right: train performance.}
       \label{fig:sidosmall}
       
       \begin{minipage}[c]{.40\linewidth}
\includegraphics[width=\linewidth]{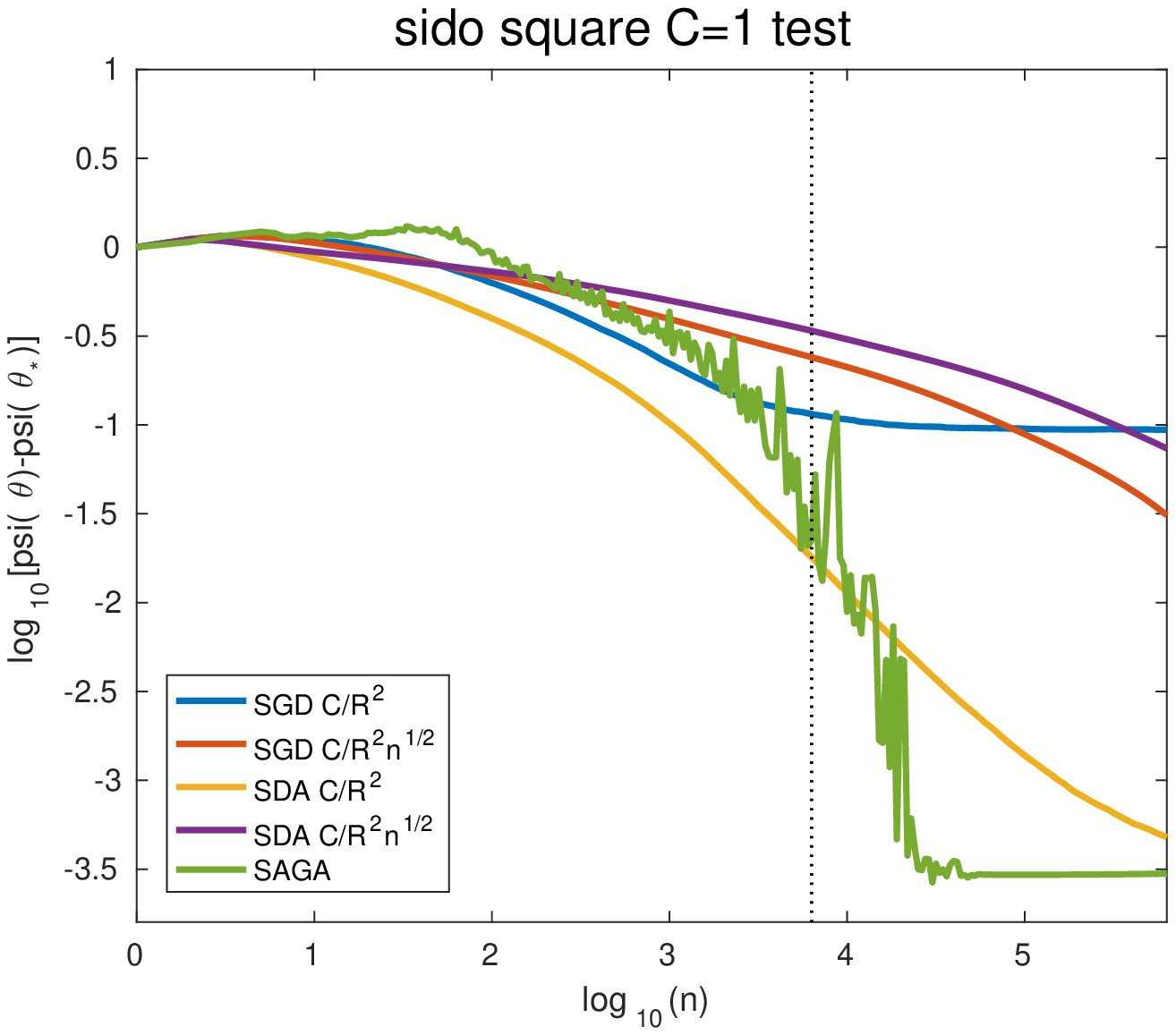}
   \end{minipage} \hspace*{.08\linewidth}
   \begin{minipage}[c]{.40\linewidth}
\includegraphics[width=\linewidth]{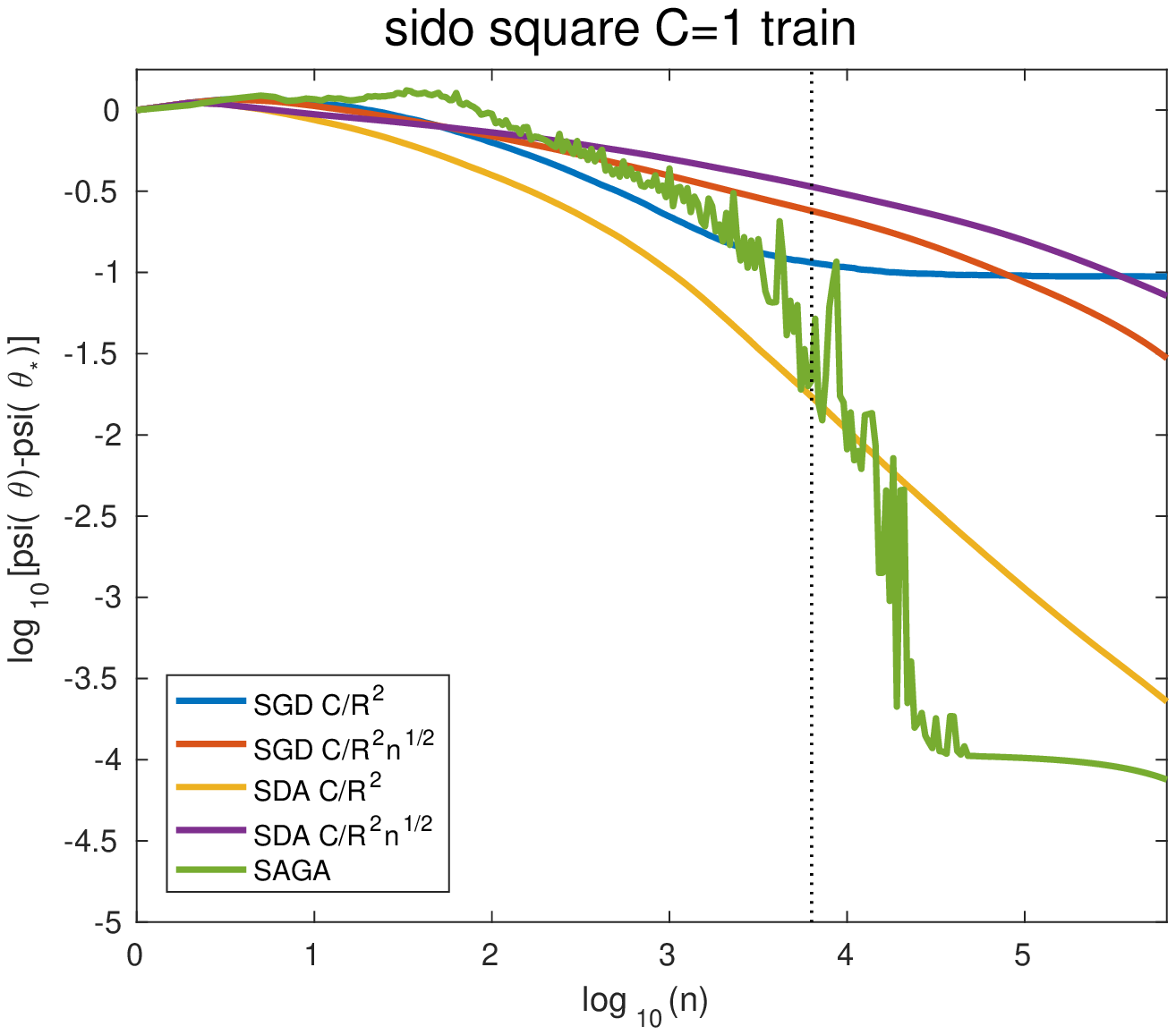}
   \end{minipage} 
 \caption{Test and train performances for $\ell_1$-regularized least-squares regression on the \emph{sido} dataset with $\lambda=256\lambda_{\text{opt}}$. 
 Left: test performance. Right: train performance.}
  \label{fig:sidobig}
\end{figure}
\end{document}